\documentclass[11pt,draft,reqno]{amsart}
\usepackage{amsmath,amssymb,amscd,amsfonts,amsthm}
\usepackage{graphics}
\usepackage{dsfont}
\usepackage[all]{xy}
\usepackage{enumerate}

{\catcode`\@=11
\gdef\n@te#1#2{\leavevmode\vadjust{%
 {\setbox\z@\hbox to\z@{\strut#1}%
  \setbox\z@\hbox{\raise\dp\strutbox\box\z@}\ht\z@=\z@\dp\z@=\z@%
  #2\box\z@}}}
\gdef\leftnote#1{\n@te{\hss#1\quad}{}}
\gdef\rightnote#1{\n@te{\quad\kern-\leftskip#1\hss}{\moveright\hsize}}
\gdef\?{\FN@\qumark}
\gdef\qumark{\ifx\next"\DN@"##1"{\leftnote{\rm##1}}\else
 \DN@{\leftnote{\rm??}}\fi{\rm??}\next@}}
\DeclareFontFamily{OT1}{wncyr}{\hyphenchar\font45 }
\DeclareFontShape{OT1}{wncyr}{m}{n}{%
   <5> <6> <7> <8> <9> gen * wncyr
   <10> <10.95> <12> <14.4> <17.28> <20.74>  <24.88>wncyr10}{}
\DeclareFontShape{OT1}{wncyr}{m}{it}{%
   <5> <6> <7> <8> <9> gen * wncyi
   <10> <10.95> <12> <14.4> <17.28> <20.74> <24.88> wncyi10}{}
\DeclareFontShape{OT1}{wncyr}{m}{sc}{%
   <5> <6> <7> <8> <9> <10> <10.95> <12> <14.4>
   <17.28> <20.74> <24.88>wncysc10}{}
\DeclareFontShape{OT1}{wncyr}{b}{n}{%
   <5> <6> <7> <8> <9> gen * wncyb
   <10> <10.95> <12> <14.4> <17.28> <20.74> <24.88>wncyb10}{}
\input cyracc.def
\def\rus{\usefont{OT1}{wncyr}{m}{n}\cyracc\fontsize{8}{10pt}\selectfont}

\DeclareMathSizes{9}{9}{7}{5} 

\theoremstyle{plain}

\newtheorem{theorem}{Theorem}

\newtheorem{lemma}{Lemma}
\newtheorem{remark}{\it Remark}
\newtheorem{corollary}{Corollary}

\theoremstyle{definition}

\newtheorem{definition}{Definition}

\newtheorem{nothing*}[theorem]{}
\newtheorem{subnothing*}[sub]{}
\newtheorem{example}{Example}

\theoremstyle{remark}

\def\An{{\bf A}^n}

\def\bA2{{\mathbf A}\!^2}

\newcommand{\dss}{\hskip -2mm\rotatebox{68}{\raisebox{-1.8\height}{\mbox{\normalsize -\hskip .1mm-\hskip .1mm-}}}\hskip -.6mm}

\newcommand{\br}{{\rm(}}

\newcommand{\bl}{{\hskip .15mm\rm)}}

\begin{document}



\title[Some
subgroups of the Cremona groups]{Some
subgroups of
the Cremona groups}

\author[Vladimir  L. Popov]{Vladimir  L. Popov${}^*$}
\address{Steklov Mathematical Institute,
Russian Academy of Sciences, Gubkina 8, Moscow\\
119991, Russia} \email{popovvl@mi.ras.ru}

\thanks{
 ${}^*$\,Supported by
 grants {\rus RFFI
11-01-00185-a}, {\rus N{SH}--5139.2012.1}, and the
program {\it Contemporary Problems of Theoretical
Mathematics} of the Russian Academy of Sciences, Branch
of Mathematics. }




\dedicatory{To M.\;Miyanishi on his $70$th birthday}

\maketitle



\begin{abstract}
We explore algebraic subgroups of the Cremona group $\mathcal C_n$ over an algebraically closed field of characteristic zero. First, we consider some
class of algebraic subgroups of $\mathcal C_n$ that we call flattenable.\;It contains all tori.\;Linearizability of
the natural rational actions
of flattenable
subgroups
on $\An$
is
intimately related to
rationality of the invariant fields and, for tori, is equivalent to
it.\;We prove
stable linearizability of these actions  and show  the existence of nonlinearizable
actions among them.\;This is applied to
exploring maximal tori in $\mathcal C_n$ and to proving the existence
of nonlinearizable, but stably linearizable elements of infinite order in
$\mathcal C_n$ for $n\geqslant 5$.\;Then we
consider some subgroups $\mathcal J(x_1,\ldots, x_n)$ of $\mathcal C_n$ that we call the rational de Jonqui\`eres subgroups.\;We prove that every
affine algebraic subgroup
of $\mathcal J(x_1,\ldots, x_n)$ is
solvable and the group of its connected components
is Abelian. We also prove that every reductive algebraic subgroup of $\mathcal J(x_1,\ldots, x_n)$ is diagonalizable. Further, we prove that the natural rational action on $\An$ of any unipotent algebraic subgroup of $\mathcal J(x_1,\ldots, x_n)$ admits a rational cross-section which is an affine subspace of $\An$.\;We show that
in this statement ``unipotent'' cannot be replaced by ``connected solvable''.\;This is applied to proving a conjecture of A.\;Joseph
on the existence of ``rational slices'' for the coadjoint representations of finite-dimensional algebraic Lie algebras $\mathfrak g$ under the assumption that the Levi decomposition of $\mathfrak g$ is a direct product.\;We then consider some overgroup $\widehat{\mathcal J}(x_1,\ldots, x_n)$ of $\mathcal J(x_1,\ldots, x_n)$ and prove that every torus in $\widehat{\mathcal J}(x_1,\ldots, x_n)$ is linearizable.\;Finally,
we prove the existence of an element $g\in \mathcal C_3$ of order $2$
such that $g\notin G$ for every connected affine algebraic subgroup $G$ of
$\mathcal C_\infty$;
in particular, $g$ is not stably linearizable.
\end{abstract}



\

\vskip 2mm

 \section{Introduction}\label{intr}

The last three decades were marked by growing interest in problems related to  the affine Cremona group
${\mathcal C}_n^{\rm aff}$\;(the group of biregular automorphisms of the affine $n$-dimensional space $\An$).\;Despite
of a remarkable prog\-ress made during these years, some fundamental problems
still remain unsolved. For instance, at the moment the linearization problem
for algebraic tori
is solved only for $n\leqslant 3$
and its difficult solution for $n=3$ is one of the highlights of the theory.

Some of these problems
may be formulated entirely in terms of group-theoretic structure of ${\mathcal C}_n^{\rm aff}$.\;Thereby, they admit
the birational counterparts related to the full Cremona group ${\mathcal C}_n$\;(the group of birational automorphisms of $\An$).\;It is of interest to explore them.\;We
have not seen
  publications purposefully developing this viewpoint.\;A step in this direction is made in this paper.

In Section \ref{flat} we first consider a class of algebraic subgroups of ${\mathcal C}_n$ that we call flattenable.\;Linearizability of their natural rational actions
on $\An$ is intimately related to
rationality of their invariant fields, the
subject of
classical Noether problem.\;All algebraic tori in $\mathcal C_n$
are contained in this class   and, for them,
these two properties,
linearizability and
rationality, are equivalent.\;We show that flattenable groups are special in the sense of Serre (see 
\cite{Ser58}) and that
every rational locally free action on $\An$ of a special group
 is stably linearizable; in particular, this is so for tori.\;On the other hand, we show that there are
 stably linearizable, but nonlinearizable rational locally free actions on $\An$ of connected affine algebraic  groups, in particular, that of tori.\;We then apply this to
the problem of describing maximal tori in ${\mathcal C}_n$ and show that nowadays one can say more
on it than in the time when Bia{\l}ynicki-Birula and Demazure wrote their papers 
\cite{Bia66}, \cite{Dem70}.\;Namely, apart from $n$-dimensional maximal tori (that are all conjugate), ${\mathcal C}_n$ for
$n\geqslant 5$ contains maximal tori of dimension
$n-3$ (and does not contain maximal tori of dimensions
 $n-2$, $n-1$ and $> n$).\;This answers a question of Hirschowitz  in
 \cite[Sect.\;3]{Hir72}.\;As
 another application, we prove the existence
of nonlinearizable, but stably linearizable elements of infinite order in
$\mathcal C_n$ for $n\geqslant 5$.

In Sections \ref{Jon} and \ref{cros}
we consider
a natural counterpart of the classical de Jonqui\`eres subgroups of ${\mathcal C}_n^{\rm aff}$ that we call
the rational de Jonqui\`eres subgroups of $\mathcal C_n$.\;We prove that their affine algebraic subgroups are solvable and have Abelian groups of connected components.\;We also prove that
reductive algebraic subgroups of the rational de Jonqui\`eres subgroups of $\mathcal C_n$ are diagonalizable.\;Then we prove that for the natural rational action on $\An$ of any unipotent algebraic subgroup of a
rational de Jonqui\`eres subgroup of $\mathcal C_n$
there exists  an affine subspace of $\An$ which is
a rational cross-section for this action (recall that
for rational actions of connected solvable affine algebraic subgroups of $\mathcal C_n$ on $\An$, the existence of {\it some}  rational cross-sections, not necessarily affine subspaces of $\An$, is ensured by the general Rosenlicht's theorem, see
\cite[Theorem\;10]{Ros56}).\;We also show that
in this result ``unipotent'' cannot be replaced by ``connected solvable''.\;We then apply this result to a conjecture of A.\;Joseph
(\cite[Sect.\,7.11]{Jos11})
on the existence of ``rational slices'' for the coadjoint representations of finite-dimensional algebraic Lie algebras $\mathfrak g$ and prove this conjecture under the assumption that the Levi decomposition of $\mathfrak g$ is a direct product.\;Further, we consider a certain natural class of overgroups of  the rational de Jonqui\`eres subgroups and, using the results of Section \ref{flat}, show that the natural action on $\An$ of any subtorus of such an overgroup is linearizable.\;Finally, we prove the existence of an element $g\in \mathcal C_3$ of order $2$ such that $g\notin G$ for every connected affine algebraic subgroup $G$ of the direct limit $\mathcal C_\infty$ of the tower of natural inclusions $\mathcal C_1\hookrightarrow\mathcal C_2\hookrightarrow\cdots$; in particular, $g$ is not stably linearizable.

\vskip 2mm

\noindent {\it Conventions, notation and some generalities}\

\vskip 1mm

\noindent
Below ``variety'' means ``algebraic variety''.\;We assume given an
algebraically closed field $k$ of characteristic zero which
serves
as
domain of definition
for each of the varieties considered below.\;Each variety is identified with its set of points rational over $k$.\;Along with the standard notation and conventions
of
\cite{Bor91} we use the following ones.
\begin{enumerate}[\hskip 5mm ---]
\item  ${\rm Aut}\,X$ is the automorphism group of
   a variety
    $X$.

  \item  ${\rm Bir}\,X$ is the group of
birational automorphisms of an irreducible
    variety\;$X$.

   \item  $X\hskip -.7mm\approx\hskip -.7mm Y$ means that $X$ and $Y$ are birationally isomorphic irreducible varieties.

        \item  If $f$ is a rational function on the product $X\times Y$ of
        varieties and  $x\in X$ is a point such that $f|^{}_{x\times Y}$ is well defined, then $f(x)$ is the element of $k(Y)$ such that $f(x,y)=f(x)(y)$ for every point $(x, y)\in X\times Y$ where $f$ is defined.

   \item
     Given a dominant rational map $\varphi\colon X\dashrightarrow Y$ of
     varieties, $\varphi^*$ is the embedding
     $k(Y)\hookrightarrow k(X)$, $f\mapsto f \circ \varphi$.

      \item  Given an action
\begin{equation}\label{action}
\alpha\colon G\times X\to X
 \end{equation}
 of a group $G$ on a set $X$ and the elements $g\in G$, $x\in X$,  then $\alpha(g, x)\in X$ is denoted by $g\cdot x$. If $H$ is a subgroup of $G$, then $\alpha|_H$ is the restriction of $\alpha$ to $H\times X$.

 \item  $\An\times{\bf A}^m$ is identified with ${\bf A}^{n+m}$ by means of the isomorphism
\begin{gather*}
\An\times{\bf A}^m\to {\bf A}^{n+m},\\[-4pt] ((a_1,\ldots, a_n), (b_1,\ldots, b_m))\mapsto (a_1,\ldots, a_n, b_1,\ldots, b_m).
\end{gather*}

  \item   $K^{\times}$ is the multiplicative  group of a field $K$.

  \item   $K^+$ is the additive group of a field  $K$.

  \item  If $K/F$ is a field extension, then $K$ is called {\it pure} (resp.\;{\it stably pure}) over $F$ if $K$ is generated over $F$ by a finite collection of algebraically independent elements (resp.\;if $K$ is contained in a field that is pure over both $K$ and $F$).

       \item  $G^0$ is the identity component of an algebraic group $G$.

  \item   ``Torus'' means ``affine algebraic torus''.
\end{enumerate}

Let  $G$ be an algebraic group and let $X$ be a variety.

If
 \eqref{action} is a morphism, then $\alpha$ is called a {\it regular action}. In this case, for every element $g\in G$, the map $X\to X$,
$x\mapsto g
\cdot x$, is an automorphism of $X$ and the image of the homomorphism $G\to {\rm Aut}\,X$, $g\mapsto\{x\mapsto g\cdot x\}$ is called an {\it algebraic subgroup of ${\rm Aut}\,X$}. A regular action $\alpha$ is called {\it locally free} if there is a dense open subset $U$ of $X$ such that the $G$-stabilizer of every point of $U$  is trivial.

From now on we assume that $X$ is irreducible.
The map
\begin{equation}\label{BA}
{\rm Bir}\,X\to {\rm Aut}_k\,k(X), \qquad \varphi\mapsto (\varphi^*)^{-1},
  \end{equation}
  is a group isomorphism.\;We always
identify ${\rm Bir}\,X$ and ${\rm Aut}_k\,k(X)$
  by means of
   \eqref{BA} when we
   speak about
    action of a subgroup of ${\rm Bir}\,X$ by $k$-automor\-phisms of $k(X)$ and, conversely, action of a subgroup of
  ${\rm Aut}_kk(X)$ by birational automorphisms of\;$X$.

Let $\theta\colon G\to {\rm Bir}\,X$ be an abstract group homomorphism.\;It determines an action of $G$ on $X$ by birational isomorphisms.
If the domain of definition of the partially defined map $G\times X\to X$, $(g, x)\mapsto \theta(g)(x)$, contains a dense open subset of $G\times X$ and coincides on it with a rational map  $\varrho\colon G\times X\dashrightarrow X$, then this action (and $\varrho$) is called a {\it rational action} of $G$ on $X$ and
$\theta(G)$ is called an {\it algebraic subgroup of ${\,\rm Bir}\,X$}.

There is a method for constructing algebraic subgroups of ${\,\rm Bir}\;X$.\;Na\-me\-ly, let $Y$ be another irreducible variety and let $\gamma\colon Y\dashrightarrow X$ be a birational isomorphism.\;Then ${\rm Bir}\,Y\to {\rm Bir}\,X$, $g\mapsto\gamma\circ g\circ \gamma^{-1}\!\!,$\;is a group isomorphism and the image of any algebraic subgroup of ${\rm Aut}\,Y$ under it is an algebraic subgroup of ${\rm Bir}\,X$.\;In fact, by
\cite[Theorem\;1]{Ros56}, this method is universal, i.e.,\;every algebraic subgroup of ${\rm Bir}\,X$ is obtained in this manner for the appropriate $Y$ and $\gamma$.\;In other words, for every rational action of $G$ on $X$ there is a regular action of $G$ on an irreducible variety $Y$, the open subsets $X_0$ and $Y_0$ of resp.\;$X$ and $Y$, and an isomorphism $Y_0\to X_0$ such that the induced field isomorphism
$k(X)=k(X_0)\to k(Y_0)=k(Y)$ is $G$-equivariant. If the action of $G$ on $Y$ is locally free, then the rational action $\varrho$ is called {\it locally free}.

Let $\varrho\colon G\times X\dashrightarrow X$ be a rational action of $G$ on $X$ and let $f$ be an element of $k(X)$.\;Then
$\{g\cdot f\mid g\in G\}$ is an ``algebraic family''
of rational functions on $X$ in the following sense: there is a rational function $\widehat f\in k(G\times X)$
such that $g\cdot f={\widehat f}(g)$
for every $g\in G$.\;Indeed, $\varrho^*(f)\in k(G\times X)$ and $\varrho^*(f)(g, x)=(g^{-1}\cdot f )(x)$
for every point $(g, x)\in G\times X$ where $\varrho^*(f)$ is defined; whence the claim.

If $X$ and $Y$ are irreducible varieties endowed with rational actions of $G$ such that there is a $G$-equivariant birational isomorphism $X\dashrightarrow Y$, then we write
\begin{equation*}
X \overset{G}{\approx} Y.
\end{equation*}

In order  to avoid a confusion, in some cases when several rational actions are simultaneously considered, we denote
 $X$ endowed with a rational action\;$\varrho$ of $G$ by
 \begin{equation*}
 {}_\varrho X.
  \end{equation*}
If $Y$ is another variety, then $X\times Y$ endowed with the rational action of $G$ via the first factor by means of $\varrho$ is denoted by $$_\varrho X\times Y.$$ We denote by
 \begin{equation*}
 _\lambda G
 \end{equation*}
the underlying variety of $G$ endowed with the action of $G$ by left translations.

If $\varrho$ is a rational action  of $G$ on $X$,  then
\begin{equation*}
\pi^{}_{G, X}\colon X\dashrightarrow X\dss G
\end{equation*}
is a rational quotient of $\varrho$, i.e.,  $X\dss G$ and $\pi^{}_{G, X}$ are resp. a variety and a
dominant rational map such that $\pi^{*}_{G, X}(k(X\dss G))=k(X)^G$ (see
\cite[Sect. 2.4]{PV94}).
Depending on the situation we
choose $X\dss G$ as a suitable variety within the class of birationally isomorphic ones.\;A {\it rational section} (resp., {\it cross-section}) for $\varrho$
is a rational map $\sigma\colon X\dss G\dashrightarrow X$ such that $\pi^{}_{G, X}\circ\sigma={\rm id}$ (resp., a subvariety $S$ of $X$ such that
$\pi^{}_{G, X}|^{}_S\colon S\dashrightarrow X\dss G$ is a birational isomorphism). The closure of the image of a rational section is a rational cross-section and,
since ${\rm char}\,k=0$, the closure of every cross-section is obtained in this manner.

The group
\begin{equation*}
{\mathcal C}_n:={\rm Aut}_k\,k(\An)
 \end{equation*}
 is called the {\it Cremona group of rank} $n$ (over $k$).\;It is endowed with a topology,
 the {\it Zariski topology of} $\mathcal C_n$,  in which families of
 elements of $\mathcal C_n$ ``algebraically paramet\-rized'' by algebraic varieties are closed,
 \cite[Sect.\;1.6]{Ser08}.\;For every algebraic subgroup $G$ of $\mathcal C_n$ and its subset $S$, the closure of $S$ in $\mathcal C_n$ coincides with the closure of $S$ in $G$ in the Zariski topology of $G$. In particular, $G$ is closed in $\mathcal C_n$.\;Left and right translations of $\mathcal C_n$ are homeomorphisms.

 We denote by  $x_1,\ldots, x_n \in k[\An]$ the {\it standard coordinate functions} on\;$\An$:
\begin{equation}\label{coord}
x_i((a_1,\ldots, a_n))=a_i.
\end{equation}
They are algebraically independent over $k$ and  $k(\An)=k(x_1,\ldots, x_n)$.\;For eve\-ry $n\geqslant 2$, we identify ${\bf A}^{n-1}$ with the
image of the embedding
${\bf A}^{n-1}\hookrightarrow \An$, $(a_1,\ldots, a_{n-1})\mapsto (a_1,\ldots, a_{n-1},0)$, and  denote the restriction $x_i|^{}_{{\bf A}^{n-1}}$ for $i=1,\ldots, n-1$ still by $x_i$. Correspondingly,
we have the embedding $\mathcal C_{n-1}\hookrightarrow {\mathcal C}_{n}$, $g\mapsto \widehat g$, where
$\widehat g\cdot x_i:=g\cdot x_i$ if $i=1,\dots, n-1$ and $\widehat g\cdot x_{n}:=x_{n}$.\;The direct limit  for the tower of these embeddings
$\mathcal C_1  \hookrightarrow\mathcal C_2\hookrightarrow\cdots\hookrightarrow\mathcal C_n\hookrightarrow\cdots$
is the {\it Cremona group $\,\mathcal C_\infty$ of infinite rank}.\;We identify every $\mathcal C_n$ with the subgroup of $\mathcal C_\infty $ by means of
the natural embedding
$\mathcal C_n\hookrightarrow\mathcal C_\infty$.
A subgroup $G$ of ${\mathcal C}_\infty$ is called {\it algebraic} if there exists an integer $n>0$ such that $G$ is an algebraic subgroup of
${\mathcal C}_n$.

We distinguish the following two algebraic subgroups of ${\mathcal C}_n$:
\begin{align*}
{\rm GL}_n&:=\{g\in \mathcal C_n\mid g\cdot x_i=\textstyle\sum_{j=1}^{n} \alpha_{ij}x_j,\;\;\alpha_{ij}\in k\},\\[-1pt]
{D}_n&:=\{g\in \mathcal C_n\mid g\cdot x_i=\alpha_ix_i,\;\;\alpha_i\in k\};
\end{align*}
$D_n$ is the maximal torus in ${\rm GL}_n$.

Let $g$ be an element and let $G$ be a subgroup of $\mathcal C_n$.\;If $g\in {\rm GL}_n$ (resp.\,$G\subseteq {\rm GL}_n$), then $g$ (resp.\;$G$) is called
a {\it linear element} (resp.\;a {\it linear subgroup}).\;If $g$ (resp.\;$G$) is conjugate to a linear element
(resp.\;a linear subgroup), then it is called a {\it linearizable element} (resp.\;a {\it linearizable subgroup}).\;If $g$ (resp.\;$G$) is a  linearizable element (resp.\;a linearizable
subgroup) of some $\mathcal C_m$ for $m\geqslant n$, then it is called a {\it stably linearizable element} (resp.\;a {\it stably linearizable subgroup}).\;A rational action $\varrho$ of an algebraic group $H$ on $\An$ is called
resp.\;a {\it linear}, {\it linearizable} or {\it stably linearizable  action}  if the image of $H$ in $\mathcal C_n$ corresponding to\,$\varrho$ is resp.\;a linear, linearizable or stably linearizable
subgroup of $\mathcal C_n$.

\vskip 2mm

{\it Acknowledgment.} I am grateful to {\sc Ming-chang Kang} who drew my attention to paper \cite{SB04}.


\section{Flattening, linearizability, tori}\label{flat}

\begin{definition}\label{emb} An affine algebraic group $G$ is called\;{\it flattenable}
if the underlying variety of $G$ endowed with the action of $G$ by left translations admits
 an equivariant open embedding into some $\An$ endowed with a rational linear
 action of $G$.\;The  $G$-module $\An$ is then called a {\it flattening} of\;$G$.
\end{definition}

Every
flattenable group is connected.


\begin{example}\label{to} $\An$ endowed with the natural action of $D_n$,
 \begin{equation}\label{diago}
{\rm diag} (\varepsilon_1,\ldots, \varepsilon_n)\cdot (a_1,\ldots,a_n):=(\varepsilon_1a_1,\ldots,
 \varepsilon_na_n),
 \end{equation}
 is a flattening of $D_n$.\;Hence, every
 torus is flattenable.
\end{example}

\begin{example}\label{gl}
The underlying vector space of  the algebra ${\rm Mat}_{n\times n}$ of all $(n\times n)$-matrices with entries in $k$ endowed with the action of ${\rm GL}_n$ by left multiplications,
$g\cdot a:=ga$, $g\in {\rm GL}_n$, $a\in {\rm Mat}_{n\times n}$, is a flattening of ${\rm GL}_n$. Hence, ${\rm GL}_n$ is flat\-ten\-able.
\end{example}

\begin{example}  Let  $G_1,\ldots, G_s$ be affine algebraic groups and let $G:=G_1\times\cdots\times G_s$. If ${\bf A}_{n_i}$ endowed with an action of $G_i$ is a flattening of $G_i$, then ${\bf A}_{n_1}\!\times\cdots\times{\bf A}_{n_s}$ endowed with the natural action of $G$
is a flattening of $G$.\;Hence, $G$ is flattenable if every
$G_i$ is.
\end{example}

\begin{example} Consider a finite-dimensional associative (not necessarily commutative) $k$-algebra $A$  with an identity element. 
The group
of all invertible elements of $A$ is then a connected affine algebraic group $G$ whose underlying variety is an open subset of that of $A$. The action of $G$ on $A$ by left multiplications is linear
and the identity map is an equivariant embedding of $G$ into $A$.\;Thus, $A$ is a flattening of $G$ and  $G$ is flattenable. If $A$ is the product of $n$ copies of the $k$-algebra $k$, we obtain Example \ref{to}. Taking $A={\rm Mat}_{n\times n}$, we obtain Example \ref{gl}.
\end{example}

In general, flattening of $G$ is not unique.

\begin{example}\label{flun} ${\rm Mat}_{n\times n}$ endowed with the action of ${\rm GL}_n$ given by
$g\cdot a:=(g^t)^{-1}a$, $g\in {\rm GL}_n$, $a\in {\rm Mat}_{n\times n}$, where $g^t$ is the transpose of
$g$, is a flattening of ${\rm GL}_n$. It is not isomorphic to that
from Example \ref{to} (as the highest weights of these two flattenings are not equal).
\end{example}

\begin{lemma}\label{cd1} If the underlying variety of a connected affine reductive algebraic group $G\neq \{e\}$ is isomorphic to an open subset $\,U$\;of $\An$,
 then $U\neq \An$ and
 the center of $\,G$
is at least one-dimensional.
\end{lemma}

\begin{proof} As $G\neq\{e\}$ is reductive,
it
contains a torus $T$
of positive dimension.\;For the action of $\,T$ on $G$ by left translations, the fixed point set is empty.\;But  for any regular action of $T$ on $\An$, the fixed point set is nonempty, see
\cite[Theorem\;1]{Bia66}. Hence, $U\neq \An$.

Since $D:=\An\setminus U\neq \varnothing$ and $U$ is affine, the dimension of every irreducible component of $D$ is $n-1$, see
\cite[Lemma\;3]{Pop72}. Since ${\rm Pic}\,\An=0$,  this entails that $D$ is the zero set of some regular function $f$ on $\An$.\;Therefore, $f|^{}_{U}$ is a nonconstant invertible regular function on $U$.
By
\cite[Theorem 3]{Ros611}, every such function is, up to a scalar multiple, a character of $G$. So
there is a nontrivial character of $G$.\;On the other hand, as $G$ is a connected reductive group,
 $G=G'\cdot C$ where $G'$ is the derived group of $G$ (it is semisimple), $C$ is the connected component  of the identity in the center of $G$ (it is a torus), and $G'\cap C$ is finite, see 
 \cite[Sect.\;14.2]{Bor91}.\;This entails that
 the character group of $G$ is a free abelian group of rank $\dim C$.\;Hence, $\dim\,C\geqslant 1$.
\quad $\square$ \renewcommand{\qed}{}\end{proof}

\begin{corollary} {\it  There are no nontrivial semisimple flattenable groups.
}
\end{corollary}

Recall from
\cite{Ser58} that an algebraic group $G$ is called {\it special} if
 every principal homogeneous space under $G$ over every field $K$ containing $k$ is trivial. By 
 \cite{Ser58} special group is automatically connected and affine.\;Special groups are classified:\;a
   connected affine algebraic group $G$ is special if and only if a maximal connected semisimple subgroup of $G$ is isomorphic to
\begin{equation}\label{special}
{\rm SL}_{n_1}\times\cdots\times {\rm SL}_{n_r}\times{\rm Sp}_{m_1}\times\cdots\times {\rm Sp}_{m_s}
\end{equation}
for some integers $r\geqslant 0, s\geqslant 0, n_i, m_j$
(by
\cite{Ser58} such groups are special, and by
\cite{Gro58} only these are).

\begin{lemma}\label{spec} Every flattenable group  $G$ is special.
\end{lemma}
\begin{proof} Let $\An$ endowed with a rational linear action $\alpha$ of $G$ be a flattening of $G$ and let  $\pi^{}_{G, {}_\alpha {\An}}\colon
{}_\alpha \An\dashrightarrow {}_\alpha {\An}\dss G$ be a rational quotient for this action.\;By Definition \ref{emb}, $\alpha$
is locally free.  Hence, by
\cite[Theorem\;1.4.3]{Pop94}, proving that $G$ is special is equivalent to proving that $\pi^{}_{G, {}_\alpha {\An}}$
 admits
a rational section.\;But
the existence of such a rational section is clear because
Definition \ref{emb} entails that ${}_\alpha\An\dss G$ is a single point.
\quad $\square$ \renewcommand{\qed}{}\end{proof}

By Lemma \ref{spec} a maximal connected semisimple subgroup of every
flattenable group is isomorphic to a group of type \eqref{special}.\;Hence, every {\it reductive}
flattenable group is a quotient $(T\times S)/C$ where $T$ is a torus, $S$ is a group of type \eqref{special} and $C$ is a finite central subgroup.

\vskip 2mm

\noindent
{\bf Conjecture.} {\it The following properties of a connected {\it reductive} algebraic group $G$
are equivalent}:
\begin{enumerate}[\hskip 3mm\rm(i)]
\item {\it $G$ is flattenable};
\item {\it $G$ is isomorphic to $
{\rm GL}_{n_1}\times\cdots\times {\rm GL}_{n_r}$.}
\end{enumerate}

\begin{theorem} \label{lle}
Let $\alpha$ be
a locally free rational action of a flattenable group $\,G$  on\;${\bf A}^m$. If
the invariant field $\,k({}_{\alpha}{\bf A}^m)^G$ is pure over $k$, then $\alpha$ is linearizable.
\end{theorem}

\begin{proof}
Consider  for $\alpha$ a rational quotient,
\begin{equation}\label{quoo}
\pi^{}_{G, {}_\alpha {\bf A}^m}\colon
{}_\alpha {\bf A}^m
\dashrightarrow
{_\alpha {\bf A}^m}\dss G.
\end{equation}
  As explained in Introduction, there is a variety $X$ endowed with a regular locally free action $\alpha'$ of $G$ such that ${}_{\alpha'}X\overset{G}{\approx}{}_\alpha{\bf A}^m$.\;By 
  \cite[Theorem 2.13]{CTKPR11}, shrinking $X$ if necessary, we may assume that the geometric quotient
  \begin{equation*}\label{torsor}
  {}_{\alpha'}X\to {}_{\alpha'}X/G
  \end{equation*} for
  $\alpha'$ exists and
  is  a torsor over ${}_{\alpha'}X/G$.\;As $G$ is special  by Lemma \ref{spec},
  this torsor is locally trivial in Zariski topology.\:Hence,
   ${}_{\alpha'}X \overset{G}{\approx}{_\lambda}G\times ({}_{\alpha'}X/G)$
   and therefore,
\begin{equation}\label{mG}
_\alpha{\bf A}^m\overset{G}{\approx}{
{_\lambda}G}\times
({_\alpha}{\bf A}^m\dss G).
\end{equation}

Let
$_\beta{\bf A}^n$
be a flattening of $G$. Definition \ref{emb} yields
\begin{equation}\label{leg}
_\lambda G \overset{G}{\approx}  {_\beta}\An.
\end{equation}
From \eqref{mG} and \eqref{leg} we obtain
\begin{equation}\label{aaa}
_\alpha{\bf A}^m \overset{G}{\approx} {_\beta}{\bf A}^n\times ({_\alpha}{\bf A}^m\dss G).
\end{equation}
The assumption of purity
and \eqref{aaa} yield
\begin{equation}\label{ra}
_\alpha{\bf A}^m\dss G\approx {\bf A}^{m-n}.
\end{equation}

Consider the action $\gamma$ of $G$ on ${\bf A}^m$ defined by
\begin{equation}\label{gamma}
_\gamma{\bf A}^m:={_\beta}{\bf A}^n\times {\bf A}^{m-n}.
\end{equation}
From \eqref{aaa}, \eqref{ra}, and
\eqref{gamma} we deduce that
\begin{equation}\label{alga}
_\alpha{\bf A}^m \overset{G}{\approx} {_\gamma}{\bf A}^m.
\end{equation}
But $\gamma$ is linear because $\beta$ is. This and \eqref{alga} complete the proof.
\quad $\square$ \renewcommand{\qed}{}\end{proof}

\begin{lemma}\label{lf}
For every
affine algebraic group $G$ and
every integer $r$ there exists a rational locally free linear action of $\,G$ on ${\bf A}^s$
for some $s>r$.
\end{lemma}
\begin{proof} By
\cite[Prop.\,1.10]{Bor91}, we may assume that $G$ is a closed subgroup of some ${\rm GL}_n$.\;As there is a closed embedding of ${\rm GL}_n$ in ${\rm GL}_{n+1}$, we may in addition assume that $n^2>r$.\;By Example 2 there is a rational locally free linear
action $\alpha$ of ${\rm GL}_n$ on ${\bf A}\!^{n\hskip -.1mm^2}$.
 Hence,
 $\alpha|^{}_G$ shares the requested pro-\break perties.
\quad $\square$ \renewcommand{\qed}{}\end{proof}

\begin{theorem}\label{specstab} Every rational locally free action $\alpha$ of a special  algebraic group $G$ on
${\bf A}^m$
is stably linearizable.
\end{theorem}
\begin{proof}
The same argument as  in the proof of Theorem \ref{lle} shows that
\eqref{mG} holds.
By Lemma \ref{lf} there is a rational locally free linear action $\gamma$ of $G$ on ${\bf A}^s$ for some $s\geqslant m$.\;Like for $\alpha$, for $\gamma$ we have
\begin{equation}\label{sG}
{_\gamma}{\bf A}^s \overset{G}{\approx}
 {_\lambda}G\times ({_\gamma}{\bf A}^s\dss G)
\end{equation}

Let
$d:=\dim G$.
Since by
\cite{Che54} the underlying variety of $G$ is rational, we have
\begin{equation}\label{grat}
G\approx {\bf A}^d.
\end{equation}
From \eqref{mG}, \eqref{sG}, and  \eqref{grat} we then obtain
\begin{equation}\label{gag}
\begin{split}
{\bf A}^m&\approx {\bf A}^d\times ({_\alpha}{\bf A}^m\dss G),\\[-3pt]
{\bf A}^s&\approx {\bf A}^d\times ({_\gamma}{\bf A}^s\dss G).
\end{split}
\end{equation}
In turn, \eqref{mG}, \eqref{sG}, and \eqref{gag} imply
\begin{equation}\label{GGG}
\begin{split}
{_\alpha}{\bf A}^m\times {\bf A}^d &\overset{G}{\approx}
{_\lambda G}\times
({_\alpha}{\bf A}^m\dss G)\times {\bf A}^d
\overset{G}{\approx}
{_\lambda G}\times{\bf A}^m,\\[-3pt]
{_\gamma}{\bf A}^s\times {\bf A}^d &\overset{G}{\approx}
{_\lambda G}\times
({_\gamma}{\bf A}^s\dss G)\times {\bf A}^d
\overset{G}{\approx}
{_\lambda G}\times{\bf A}^s.
\end{split}
\end{equation}

Since $s\geqslant m$, we have $d+s-m\geqslant 0$ and
 from \eqref{GGG}  we deduce
\begin{equation}\label{fini}
\begin{split}
{_\alpha}{\bf A}^m\times {\bf A}^{d+s-m}&=
{_\alpha}{\bf A}^m\times{\bf A}^d\times {\bf A}^{s-m}\\[-3pt]
&\overset{G}{\approx}
{_\lambda G}\times {\bf A}^m\times {\bf A}^{s-m}
\\[-3pt]
&
=
{_\lambda G}\times {\bf A}^s\\[-3pt]
&\overset{G}{\approx}
{_\gamma}{\bf A}^s\times {\bf A}^d.
\end{split}
\end{equation}
Since the action of $G$ on ${_\gamma}{\bf A}^s\times {\bf A}^d$
is linear, \eqref{fini} completes the proof.
\quad $\square$ \renewcommand{\qed}{}\end{proof}

The next theorem implies that
 ``stably linearizable''
 in Theorem \ref{specstab} cannot be replaced by
 ``linearizable''.

\begin{theorem} For every connected semisimple algebraic group $\,G\neq \{e\}$,
  there
exists a rational nonlinearizable locally free action of $\,G$ on ${\bf A}^d$ for $d={\dim\,G}$.
\end{theorem}
\begin{proof}
Since
\eqref{grat} holds,
there exists a rational
locally free action $\alpha$ of $G$ on ${\bf A}^d$ such that
$_\lambda G\overset{G}{\approx}{}_\alpha {\bf A}^d$.\;We claim that $\alpha$ is nonlinearizable.\;For, otherwise, we would get a rational locally free linear (hence, regular) action of $G$ on ${\bf A}^d$.\;Since $d=\dim\,G$,
 one of
 its orbits
 is open in ${\bf A}^d$ and isomorphic to the underlying variety of $G$.\;Therefore, by Lemma \ref{cd1}
the center of $G$ is at least one-dimensional\,---\,a contradiction because $G$ is semisimple.
\quad $\square$ \renewcommand{\qed}{}\end{proof}

For tori  we can get an additional information.

\begin{lemma}\label{lfat} Let $X$ be an irreducible variety endowed with a rational faithful action $\alpha$ of  a  torus $\,T$.\;Then
\begin{enumerate}[\hskip 9mm]
\item[$({\rm i})$] $\alpha$ is locally free{\rm ;}
\item[$({\rm ii})$] $\dim T\leqslant \dim X${\rm;}
\item[$({\rm iii})$] ${\rm tr\,deg}_k k(X)^T=\dim X-\dim T$.
\end{enumerate}
\end{lemma}

\begin{proof}
   By
   \cite[Cor.\,2 of Lemma\,8]{Sum74} (see also
   \cite[Cor.\,1 of Prop.\,1]{Bia66}) there is an irreducible affine variety $Y$ endowed with a regular action of $\,T$ such\;that
$X\overset{T}{\approx} Y$.
  Hence, we may (and shall) assume that $X$ is affine and $\alpha$ is regular.
  By
  \cite[Theorem\;1.5]{PV94},
    we also may (and shall) assume that $X$ is a closed $T$-stable subset of a finite-dimensional algebraic $T$-module $V$  not contained in a proper $T$-submodule of $V$.

    As $\alpha$ is faithful, the kernel of the action of $T$ on $V$ is trivial.\;As $T$ is a torus, $V$ is the direct sum of $T$-weight subspaces.\;Hence, if $U$ is the
   complement in $V$ to the union of these spaces,
   this kernel coincides with
   the $T$-stabilizer of every point of $U$.\;Thus, these stabilizers are trivial.\;But by the construction,
   $X\cap U$ is a nonempty open subset of $X$.\;This proves (i) that, in turn, entails (ii) and, by   
   \cite[Cor.\;in\;Sect.\;2.3]{PV94}, also (iii).
\quad $\square$ \renewcommand{\qed}{}\end{proof}

\begin{corollary}[
{\cite{Dem70}}]\label{Tn} The dimension of every  torus in ${\mathcal C}_n$ is at most\;$n$.
\end{corollary}

\begin{corollary}\label{stT} Every rational action of a   torus on
$\An$ is stably linearizable.
\end{corollary}
\begin{proof}  Since tori are special groups, this follows from Lemma \ref{lfat}(i) and Theorem \ref{specstab}.
\quad $\square$ \renewcommand{\qed}{}\end{proof}

\begin{theorem} \label{torus}
The following properties of a rational action $\alpha$
of a   torus $\,T$ on
$\An$ are equi\-va\-lent{\rm:}
\begin{enumerate}[\hskip 9mm]
\item[$({\rm i})$]
$\alpha$ is
linearizable{\rm;}
\item[$({\rm ii})$] the invariant field $k(_\alpha\An)^T$\;is pure over $k$.
\end{enumerate}
\end{theorem}
\begin{proof}
      Assume that (ii) holds.\;Let $T_0$ be the kernel of the action of $T$ on $X$.\;By Lemma\;\ref{lfat},  the induced action of $T/T_0$ on $\An$ is locally free. Hence, replacing $T$ by $T/T_0$, we may
       assume that the action of $T$ on $\An$ is locally free.\;Since $T$ is flattenable, in this case   (ii)$\Rightarrow$(i) follows from
      Theorem \ref{lle}.

      (i)$\Rightarrow$(ii) is the corollary of the following more general statement.

\begin{lemma}\label{diag} For any rational linear action $\alpha$ of a diagonalizable affine algebraic group $D$ on $\An$, the invariant field $k(_\alpha\An)^D$  is pure over $k$.
\end{lemma}
\begin{proof}[Proof of Lemma {\rm \ref{diag}}] By
\cite[Prop.\;8.2(d)]{Bor91},
the image of $D$ under the homomorphism $D\to {\rm GL}_n$ determined by $\alpha$ is conjugate to a subgroup of $D_n$.\;Hence, we may (and shall) assume that $D$ is a closed subgroup of $D_n$.\;Since $\An$ with the natural action of $D_n$ is a flattening of $D_n$ (see Example \ref{emb}), we have ${_\alpha}\An\overset{D}{\approx}{_\lambda}D_n$. Therefore,
\begin{equation}\label{trr}
{_\alpha}\An\dss D\approx D_n/D.
 \end{equation}
 But $D_n/D$ is a torus, see
 \cite[Props.\;8.4 and 8.5]{Bor91}, hence, a rational variety.\;The claim now follows from \eqref{trr}.
\quad $\square$ \renewcommand{\qed}{}\end{proof}
\renewcommand{\qed}{}\end{proof}

\begin{corollary}[
{\cite[Cor.\,2 of Prop.\,1]{Bia66}}]\label{torii}
 \

\begin{enumerate}[\hskip 9mm]
\item[$({\rm a})$]  Every faithful rational action of a torus $\,T$  on $\An$ is linearizable in either of the following cases{\rm:}
\begin{enumerate}[\hskip 9mm]
\item[$({\rm i})$] $\dim T\geqslant n-2${\rm;}
\item[$({\rm ii})$]
$\;n\leqslant 3$.
\end{enumerate}

 \item[$({\rm b})$]
 Every $d$-dimensional torus  in $\,{\mathcal C}_n$
   for $d=n-2, n-1, n$ is conjugate to a subgroup of $D_n$.\;In particular,  every $n$-dimensional torus in ${\mathcal C}_n$ is conjugate to
 $D_n$.
\end{enumerate}
\end{corollary}

\begin{proof} (a) By Corollary \ref{Tn},
if $\,T\neq\{e\}$, then (ii)$\Rightarrow$(i).\;Assume that (i) holds. Then ${\rm tr\,deg}_k k(\An)^T\leqslant 2$ by
Lemma \ref{lfat}(iii).\;As $k(\An)^T$ is uni\-rational, it is then pure over $k$ by the L\"uroth and Castelnuovo theorems; whence the claim by Theorem \ref{torus}.

Part (b) follows from (a).
\quad $\square$ \renewcommand{\qed}{}\end{proof}

By Corollaries \ref{Tn} and \ref{torii}(b) all $n$-dimensional tori in $\mathcal C_n$ are maximal and conjugate and there are no maximal $(n-1)$- and $(n-2)$-dimensional tori
in $\mathcal C_n$.
\;In dimension $n-3$ the situation is different:

\begin{theorem}\label{nr}
Let $n\geqslant 5$.\;Every $(n-3)$-dimensional connected affine algebraic group $G$
  can be realized as an algebraic subgroup of $\,\mathcal C_n$ such that
  \begin{enumerate}[\hskip 9mm]
  \item[$({\rm i})$] $k(\An)^G$ is not pure, but stably pure over $k${\rm;}
  \item[$({\rm ii})$] the natural rational action of $\,G$ on $\An$ is locally free.
  \end{enumerate}
\end{theorem}
\begin{proof} By
\cite{SB04}  there exists a nonrational threefold $X$ such that
${\bf A}^2\times X \approx {\bf A}^5$.\;Then
$\An\approx{\bf A}^{n-3}\times X$.\;This and  \eqref{grat} yield that
there exists a rational locally free action $\gamma$ of $G$ on $\An$ such that ${_\gamma}\An\overset{G}\approx{_\lambda}G\times X$. Since $k({_\lambda}G)^G=k$, by 
\cite[Lemma 3]{Ros612}, we have
${_\gamma}\An\dss G\approx X$; whence\linebreak the claim.
\quad $\square$ \renewcommand{\qed}{}\end{proof}

\begin{corollary}\label{non}  Let $n\geqslant 5$.\;Then
\begin{enumerate}[\hskip 9mm]
\item[$({\rm a})$] there is a  rational locally free nonlinearizable action  of an $(n-3)$-dimensional torus on\;$\An${\rm;}
\item[$({\rm b})$]  $\,{\mathcal C}_n$ contains an $(n-3)$-dimensional maximal   torus.
\end{enumerate}
\end{corollary}
\begin{proof} Use the notation of Theorem \ref{nr} and its proof and let $G$ be a torus. Then $\gamma$ is nonlinearizable by  Theorem  \ref{torus}. This proves (a).\;As the torus $G$ is
not conjugate to a subgroup of $D_n$, Corollary  \ref{torii}(b) implies that it is maximal.\;This proves (b).
\quad $\square$ \renewcommand{\qed}{}\end{proof}

\begin{corollary}\label{g}  Every  $\,{\mathcal C}_n$ for $n\geqslant 5$ contains a nonlinearizable, but stably linearizable element
of infinite order.\footnote{In the original version of this paper (see {\tt arXiv:1110.2410v1-v3}) the inequality $n\geqslant 5$ in Theorem \ref{nr} and Corollaries \ref{non}, \ref{g}, and \ref{fifi}
was replaced by $n\geqslant 6$ because  I used the result of \cite{BCSS85} in place of that of \cite{SB04},
about which I was unaware.}
\end{corollary}

\begin{proof} For any subset $X$ of $\mathcal C_n$ denote by $\overline{X}$ the closure of $X$ in the Zariski topology of $\mathcal C_n$ (see Section 1).\;By Corollary \ref{non}(b),  $\mathcal C_n$ contains an $(n-3)$-dimensional maximal  torus $T$.\;By
\cite[Sect.\;III.8.8]{Bor91},
 there exists an element $g\in T$ such that $T=\overline S$ for $S:=\{g^d\mid d\in \mathbf Z\}$.\;Corollary\;\ref{stT} yields that
$g$ is stably linearizable.\;Assume that $g$ is linearizable and let $h\in\mathcal C_n$  be an element  such that
$hgh^{-1}\in D_n$.\;Then $S\subset h^{-1}D_nh$.\;Since left and right translations of $\mathcal C_n$ are homeomorphisms and $\overline{D_n}=D_n$, we obtain $$T=\overline{S}\subset\overline{h^{-1}D_nh}=h^{-1}\overline{D_n}h=h^{-1}D_nh.$$
This contradicts the maximality of $T$ because $h^{-1}D_nh$ is an $n$-dimensional torus.
 \quad $\square$ \renewcommand{\qed}{}\end{proof}

The next statement yields a rectification of Corollaries \ref{stT} and \ref{g}.

\begin{theorem} Every   torus $\,T$ in ${\mathcal C}_m$ is conjugate in ${\mathcal C}_{m+\dim T}$  to a subgroup of $\,D_{m+\dim T}$.
\end{theorem}

\begin{proof} Let $\alpha$ be the natural rational action of $T$ on ${\bf A}^m$ and
 let $d:=\dim T$. By Lemma \ref{lfat}, $\alpha$ is locally free.\;By
 \cite[Lemma 3]{Ros612}, \eqref{mG} and \eqref{grat} we have
\begin{equation}\label{bq}
\begin{split}
({\bf A}^d\times{}_\alpha{\bf A}^m)\dss T&\approx{\bf A}^d\times({_\alpha}{\bf A}^m\dss T),\\[-3pt]
{_\alpha}{\bf A}^m&\overset{T}{\approx} {
_\lambda }T\times ({_\alpha}{\bf A}^m\dss T),\\[-3pt]
T&\approx {\bf A}^d.
\end{split}
\end{equation}
From \eqref{bq} we deduce that $k({}_\alpha{\bf A}^m\times{\bf A}^d)^T$ is pure over $k$.\;Since $\,T$ is flattenable, Theo\-rem\;\ref{lle} then entails that
${\bf A}^d\times{_\alpha}{\bf A}^m \overset{T}{\approx}{_\gamma}{\bf A}^{m+d}$
for a rational linear action $\gamma$; whence the claim.
\quad $\square$ \renewcommand{\qed}{}\end{proof}


\section{Subgroups of the rational de Jonqui\`eres groups}\label{Jon}

Let $t_1,\ldots, t_n$ be a system of generators of $k(\An)$ over $k$,
\begin{equation*}
k(\An)=k(t_1,\ldots, t_n).
\end{equation*}
The elements $t_1,\ldots, t_n$ are algebraically independent over $k$ and determine the following flag of subfields of $k(\An)$:
\begin{gather}\label{ff}
\begin{gathered}
K_{n}\!\subset\! K_{n-1}\!\subset\!\!\cdots\! \!\subset\!K_0,
\mbox{where}\;\;
K_i\!:=\!\begin{cases}\! k(t_{i+1}, t_{i+2},\ldots, t_n) &\hskip -2mm\mbox{if $\,i\!\leqslant\! n\!-\!1$},\\
k &\hskip -2.5mm\mbox{\,if $i\!=\!n$}.
\end{cases}
\end{gathered}
\end{gather}

For any elements $f_i\in K_i$ and $\mu_i\in k^{\times}$, $i=1,\ldots, n$, put
\begin{equation}\label{trio}
t'_i:=\mu_it_i+f_i \quad\mbox{and}\;\; K'_i:=\begin{cases}\! k(t'_{i+1}, t'_{i+2},\ldots, t'_n) &\mbox{if $\;i\leqslant n-1$},\\
k &\mbox{if $\;i=n$}.
\end{cases}
\end{equation}
It follows from \eqref{trio} that there are elements $f'_i\in K'_i$, $i=1,\ldots, n$, such that
\begin{equation*}
t_i=\mu^{-1}_it'_i+f'_i.
\end{equation*}
Hence, $K'_i=K_i$ for every $i$.\;In particular, $t'_1,\ldots, t'_n$ is an algebraically independent  system of gene\-ra\-tors
of $k(\An)$ over $k$, so there is an element $g\in {\mathcal C}_n$ such that
\begin{equation}\label{T}
g\cdot t_i=\mu_it_i+f_i\qquad \mbox{for every} \quad i=1,\ldots,n.
\end{equation}

The set $\mathcal J(t_1,\ldots, t_n)$ of all such elements $g$
is a subgroup of ${\mathcal C}_n$.\;It stabilizes the flag of subfields \eqref{ff}:
\begin{equation}\label{ifl}
g\cdot K_i= K_i\qquad\mbox{for all}\;\; g\in \mathcal J(t_1,\ldots, t_n)\;\;\mbox{and}\;\; i=0,\ldots, n.
\end{equation}

If $s_1,\ldots, s_n$ is another system of generators of of $k(\An)$ over $k$, then the subgroups
$\mathcal J(t_1,\ldots, t_n)$ and $\mathcal J(s_1,\ldots, s_n)$
are conjugate in ${\mathcal C}_n$.

Given an analogy of the construction of $\mathcal J(t_1,\ldots, t_n)$ with that of the de Jonqui\`eres subgroup of ${\rm Aut}_kk[t_1,\ldots, t_n]$,\;cf.
\cite[p.\,85]{vdE00}, we call
$\mathcal J(t_1,\ldots\linebreak \ldots, t_n)$
 the {\it rational de Jonqui\`eres subgroup} of $\,{\mathcal C}_n$ with respect to $t_1,\ldots, t_n$.

 \begin{example}\label{LK} By the Lie--Kolchin theorem every closed connected solvable subgroup $G$ of
${\rm GL}_n$
is conjugate in ${\rm GL}_n$ to a subgroup of
$\mathcal J(x_1,\ldots, x_n)$. Hence, $G$
lies in $\mathcal J(t_1,\ldots, t_n)$ where $t_1,\ldots, t_n$ are
 the homogeneous linear forms in $x_1,\ldots, x_n$.
\end{example}

In the notation of \eqref{T}, for every $i=1,\ldots, n$, we have the following maps:
\begin{equation}\label{cF}
\begin{split}
\chi_i\colon &\mathcal J(t_1,\ldots, t_n)\to k^{\times},\quad g\mapsto \mu_i,\\[-2pt]
\varphi_i\colon &\mathcal J(t_1,\ldots, t_n)\to K_i,\quad g\mapsto f_i.
\end{split}
\end{equation}

\begin{lemma}\label{Jo}
For every $i=1,\ldots,n$,
\begin{enumerate}[\hskip 9mm]
\item[$({\rm a})$] $\chi_i$
is a homomorphism of groups{\rm;}
\item[$({\rm b})$]  for all $g_1, g_2\in {\mathcal J}(t_1,\ldots, t_n),$
\begin{equation}\label{coc}
\varphi_i(g_1g_2)=\chi_i(g_2)\varphi_i(g_1)+g_1\!\cdot\!(\varphi_i(g_2));
\end{equation}
\item[$({\rm c})$] if $\,G$ is an algebraic subgroup of $\,{\mathcal C}_n$ contained in ${\mathcal J}(t_1,\ldots, t_n)$, then
$\chi_i|^{}_{G}$ is a regular function on $\,G$ and there is a rational function $F_i
\in k(G\times \An)$ such that
$F_i(g)=\varphi_i(g)$ for all $g\in G${\rm;}
\item[$({\rm d})$]  the order of every element $g\in\bigcap_{i=1}^n{\rm ker}\,\chi_i$, $g\neq e$, is infinite.
\end{enumerate}
\end{lemma}
\begin{proof} Let $g_1, g_2\in {\mathcal J}(t_1,\ldots, t_n)$. Then \eqref{T} and \eqref{cF} yield
\begin{equation}\label{chaa}
\begin{split}
\chi_i(g_1g_2)t_i+\varphi_i(g_1g_2)&=g_1g_2\cdot t_i=g_1\cdot (g_2\cdot t_i)\\[-1pt]
&=g_1\cdot \big(\chi_i(g_2) t_i+\varphi_i(g_2)\big)\\[-1pt]
&=\chi_i(g_2)\big(\chi_i(g_1)t_i+\varphi_i(g_1)\big)+g_1\cdot \big(\varphi_i(g_2)\big).
\end{split}
\end{equation}
As the image of $\varphi_i$ lies in  the ${\mathcal J}(t_1,\ldots, t_n)$-stable field $K_i$,
\eqref{chaa} and algebraic in\-de\-pen\-dence of $t_1,\ldots, t_n$ over $k$ yield that
\eqref{coc} and $\chi_i(g_1g_2)=\chi_i(g_1)\chi_i(g_2)$ hold.
This proves (a) and (b).

(c) Let $\alpha\colon G\times \An\dashrightarrow \An$ be the natural rational action of $G$ on $\An$ and let $\beta\colon G\!\times\!\An\to G\!\times\!\An$,
$(g, a)\mapsto (g^{-1}, a)$.\;Put $S_i\!:=\!\beta^*(\alpha^*(t_i))\!\in\! k(G\!\times \!\An)$. Then $S_i(g, a)=
t_i(\alpha(\beta(g,a)))=t_i(\alpha(g^{-1},a))=
t_i(g^{-1}\cdot a)=(g\cdot t_i)(a)$ for every $(g,a)$ in the domain of definition.\;Hence, $S_i(g)=\chi_i(g)t_i+\varphi_i(g)$ for every $g\in G$.
Given that $S_i\in
k(G\times\An)=k(G)(t_1,\ldots, t_n)$ and
$\varphi_i(g)\in K_i$, this implies (c).

(d) As $g\neq e$ and $\chi_i(g)=1$ for every $i$, \eqref{T} and \eqref{cF} entail that there is $j$ such that
$\varphi_j(g)\neq 0$.\;Let $d$ be the largest $j$ with this property. Then $g\cdot f=f$ for every $f\in K_d$. As $g\cdot t_d=t_d+\varphi_d(g)$ and $\varphi_d(g)\in K_d$, this yields
\begin{equation}\label{oooo}
g^s\cdot t_d=t_d+s\varphi_d(g)\quad\mbox{for every}\;\;s\in \bf Z.
\end{equation}
 Since $\varphi_d(g)\neq 0$ and ${\rm char}\,k=0$, \eqref{oooo} implies that $g^s\neq e$ for every $s\neq 0$.\;This proves (d).
\quad $\square$ \renewcommand{\qed}{}\end{proof}

\begin{theorem} \label{solva}
Let $G$ be an affine
algebraic subgroup of
$\mathcal J(t_1,\ldots, t_n)$.\;Then $G$ is solv\-able and $\,G/G^0$ is Abelian.
\end{theorem}
\begin{proof}
First, consider the case where $G$ is finite; we then have to prove that $G$ is Abelian.
Consider the
homomorphism
\begin{equation*}\label{hJo}
\delta\colon \mathcal J(t_1,\ldots, t_n)\to D_n,\quad g\mapsto{\rm diag}(\chi_1(g),\ldots, \chi_n(g)).
\end{equation*}
Since ${\rm ker}\,\delta=\bigcap_{i=1}^n{\rm ker}\,\chi_i$ and $G$ has no elements of infinite order,
Lemma \ref{Jo}(d) implies that
$G\cap{\rm ker}\,\delta
=\{e\}$.\;Therefore,
$\delta$ embeds $G$ into the Abelian group $D_n$; whence, the claim.

Now consider the general case.\;By
\cite[Lemma 5.11]{BS64}, there is a finite subgroup $H$ of $G$ that intersects every connected component of $G$.\;Hence, the restriction to $H$ of the canonical homomorphism $G\to G/G^0$ is a surjective homomorphism $H\to G/G^0$.\;According
 to what we have already proved, $H$ is Abelian.\;This shows that $G/G^0$ is Abelian.\;By 
 \cite[Theorem\;9.2.5]{Hal59}, the problem is then reduced to proving that $G^0$ is solvable.

 Since ${\rm char}\,k=0$, there exists  a Levi subgroup $L$  in $G^0$, see
 \cite[11.22]{Bor91}.\;It is  a connected reductive group and we have to show that $L$ is a torus, i.e., that the
derived subgroup $L'$ of $L$ is trivial.\;Arguing on the contrary, assume that $L'\neq \{e\}$.\;Then
$L'$ contains an element $g\neq e$ of finite order.\;Indeed,
$L'$ contains a torus $\neq\!\{e\}$ (see
\cite[Cor.\;2 in Sect.\;IV.13.17 and Theorem 12.1(b)]{Bor91}), but every torus $\neq\!\{e\}$ has a nontrivial torsion (see
\cite[Prop.\;8.9(d)]{Bor91}).\;On the other hand, $L'\subseteq \bigcap_{i=1}^n{\rm ker}\,\chi_i$ as every homomorphism $L\to k^{\times}$ contains $L'$ in its kernel.\;By Lemma \ref{Jo}(d) this entails that the order of $g$ is infinite.\;This contradiction completes the\break  proof.
\quad $\square$ \renewcommand{\qed}{}\end{proof}

\begin{corollary}\label{FFFF}
Every finite
subgroup
of
$\,\mathcal J(t_1,\ldots, t_n)$ is Abelian.
\end{corollary}

\begin{theorem} Let $\,G$ be a reductive algebraic subgroup of $\,\mathcal J(t_1,\ldots, t_n)$. Then $\,G$ is a diagonalizable group.
\end{theorem}
\begin{proof}
By Theorem \ref{solva} the reductive group $G^0$ is solvable.\;Hence, $G^0$ is a torus.\;Let $H$ be the subgroup of $G$ from the proof of Theorem \ref{solva}.\;It acts on $G^0$ by conjugation because $G^0$ is normal in $G$.\;The fixed point set $F$ of this action is a closed subgroup  of $G^0$.\;Assume that $F\neq G^0$.\;Then, since the torsion subgroup of $G^0$ is dense in $G^0$ (see
\cite[Cor.\;III.8.9]{Bor91}), there exists an element $g\in G^0\setminus F$ whose order is finite.\;Let $S$ be the subgroup of $G^0$ generated by the set $\{hgh^{-1}\mid h\in H\}$.\;Since $G^0$ is Abelian and the orders of $g$ and $H$ are finite, $S$ is finite as well.\;Since $S$ is stable with respect to the action of $H$ on $G^0$ by conjugation, this implies that the subgroup generated by $S$ and $H$ is finite, too.\;Corollary \ref{FFFF} then yields that
this subgroup is Abelian.\;Hence, $g\in F$\,---\,a contradiction.\;Therefore, $F=G^0$, i.e., $G^0$ and $H$ commute.\;Since the Abelian groups $H$ and $G^0$ generate $G$, this implies that $G$ is Abelain.\;The claim then follows  by
\cite[Prop.\,III.8.4(4) and Cor.\;III.4.4(1)]{Bor91}.\quad $\square$ \renewcommand{\qed}{}\end{proof}

\section{Affine subspaces as cross-sections}\label{cros}

By
\cite[Theorem\;10]{Ros56}, for every rational action of a connected solvable algebraic group
there exists a rational cross-section.\;The next theorem refines this for some rational actions on $\An$ by showing that there exist cross-sections that are affine subspaces of\;$\An$.

\begin{theorem}\label{unitri}
Let $G\neq \{e\}$ be a unipotent affine algebraic subgroup of the group $\mathcal J(x_1, \ldots , x_n)$ and let $\alpha$ be the corresponding rational action of $\,G$ on $\An$.\;Then there exist  a sequence $1\leqslant i_1<\cdots<i_m\leqslant n$ of natural numbers and a sequence $\Theta_1,\ldots, \Theta_m$ of nonempty open subsets  of $\,k$\;such that
for every $(c_1,\ldots, c_m)\in \Theta_1\times\cdots\times\Theta_m$ the affine subspace of $\An$ defined by the equations $\br$see \eqref{coord}$\bl{\rm:}$
 \begin{equation*}
 x_{i_1}=c_1, \ldots, x_{i_m}=c_m,
 \end{equation*}
 is a rational cross-section for $\alpha$.
\end{theorem}

For the proof of Theorem \ref{unitri} we need the following

\begin{lemma}\label{add} Let $K$ be a field of characteristic $0$ and let $f(x)$ be a rational function in a variable $x$ with the coefficients in $K$.\;Let $K'$ be a subfield of $K$.\;If
\begin{equation}\label{ab}
f(a_1+a_2)=f(a_1)+f(a_2)
\end{equation}
whenever $f$ is defined at $a_1, a_2$ and $a_1+a_2\in K'$, then there is an element
$c\in K$ such that $f(x)=cx$.
\end{lemma}
\begin{proof}[Proof of Lemma {\rm \ref{add}}] We may (and shall) assume that $f\neq 0$.
Let $\overline{K}$ be an algebraic closure of $K$.\;First, we claim that \eqref{ab} holds
whenever $f$ is defined at $a, b$ and $a+b\in \overline{K}$. Indeed,
by \eqref{ab} the rational function
$F(x_1, x_2):=f(x_1)+f(x_2)-f(x_1+x_2)$ (see \eqref{coord})
vanishes at every point of ${\bf A}^2(K')$ where it is defined.\;Since ${\bf A}^2(K')$ is Zariski dense
in ${\bf A}^2$, this yields $F=0$; whence the claim.

Thus,  $f$
is a rational partially defined endomorphism of the algebraic group $\overline{K}^+\!.$\;But by 
\cite{Wei55} (cf.\;also
\cite[Sect.\;11.1.1]{Mer80}) every rational partially defined homomorphism of algebraic groups is, in fact, an everywhere defined algebraic homomorphism.\;This entails that $f(x)\in K[x]$.\;Since $f$ has only finitely many roots, ${\rm ker}f$ is finite.\;Therefore, $f(\overline{K}^+)$ is a one-dimensional closed subgroup of $\overline{K}^+$; whence $f(\overline{K}^+)=\overline{K}^+$.\;On the other hand,  since ${\rm char}\,K=0$, there are no nonzero elements of finite order in $\overline{K}^+$.\;Hence, ${\rm ker}\,f=\{0\}$.\;Thus, $f$ is an isomorphism; whence
the claim.
\quad $\square$ \renewcommand{\qed}{}\end{proof}

\begin{proof}[Proof of Theorem {\rm\ref{unitri}}] We shall use the notation of
\eqref{ff},
\eqref{cF} with
\begin{equation*}
t_1=x_1,\ldots, t_n=x_n.
\end{equation*}

Since  ${\rm char}\,k=0$, $G$ is connected.\;As $G$ is a nontrivial unipotent group,
it contains a one-dimensional normal subgroup $U$
isomorphic to $k^{+}$.\;We  identify
$U$ with $k^+$ by an isomorphism $U\to k^+$.\;Since $G$ is unipotent, there are no nontrivial algebraic homomorphisms $G\to k^{\times}$,
therefore, by Lem\-ma\;\ref{Jo} there are rational functions $F_i\in k(G\times \An)$ such that
\begin{equation}\label{U}
\begin{split}
g\cdot x_i &=x_i + F_i(g),\\
F_i(g)&\in K_i,
\end{split}
\qquad \mbox{for every $g\in G$ and $i$}.
\end{equation}

 Since $U\neq \{e\}$, \eqref{U} entails that $F_j(u)\neq 0$ for some $u\in U$ and $j$. Let $d$ be the largest $j$ appearing in this fashion. Then \eqref{U} and \eqref{ff} yield
\begin{equation}\label{Ud}
K_d^U=K_d.
\end{equation}
In turn, from \eqref{Ud} and \eqref{coc} we infer that
\begin{equation*}\label{charr}
F_d(u_1+u_2)=F_d(u_1)+F_d(u_2)\qquad \mbox{for all\; $u_1, u_2\in U$}.
\end{equation*}
By Lemma \ref{add},
this implies that
there is a nonzero element $s\in K_{d}$ such that
\begin{equation}\label{fd}
F_d(u)=us\qquad\mbox{for every \;$u\in U$}.
\end{equation}
Thus, by \eqref{U} and \eqref{fd},
\begin{equation}\label{ddddd}
u\cdot x_d=x_d + us,\qquad\mbox{for every \;$u\in U$}.
\end{equation}

By
\cite[Theorem 1]{Ros56}, there exists a nonempty open subset $\An_0$ of $\An$ and its embedding in an irreducible variety $Y$,
\begin{equation*}
Y\hookleftarrow \An_0\subseteq \An,
\end{equation*}
such that the rational action of $\,U$ on $Y$ determined by $\alpha|_U$ and by this embedding is regular.\;We identify $\An_0$ with the image of this embedding.\;By
\cite[Theorem 2]{Ros56}, shrinking $Y$ if necessary, we may (and shall) assume that
 there exists a geometric quotient of $Y$ by this action of $U$,
 \begin{equation*}\label{pii}
 \pi^{}_{U, Y}\colon Y\to Y/U.
 \end{equation*}
 Then $\pi^{}_{U, Y}|^{}_{\An_0}$ is the restriction to $\An_0$ of
 a rational quotient for $\alpha|_U$,
   \begin{equation*}
   \pi^{}_{U,\An}\colon \An\dashrightarrow Y/U=:\An\dss U.
    \end{equation*}

    Let $H:=G/U$. Then $\alpha$ induces a rational action $\beta$ of $H$ on $Y/U$. Consider a rational quotient for $\beta$,
\begin{equation*}
\pi^{}_{H, Y/U}\colon Y/U\dashrightarrow (Y/U)\dss H.
\end{equation*}
  Then the composition
 \begin{equation*}
 \pi^{}_{G,\An}:=\pi^{}_{H, Y/U}\circ \pi^{}_{U,\An}
  \end{equation*}
  is a rational quotient for $\alpha$,
   \begin{equation*}
   \pi^{}_{G,\An}\colon \An\dashrightarrow (Y/U)\dss H=:\An\dss G.
    \end{equation*}

Shrinking $\An_0$ and $Y$ if necessary, we may (and shall) assume that
\begin{enumerate}[\hskip 3mm\rm(i)]
\item  $\pi^{}_{H, Y/U}$ is a morphism;
\item $s|^{}_{\An_0}$
is regular and vanishes nowhere.
\end{enumerate}

    To sum up, we have the following commutative diagram:
     \begin{equation}\label{diagra}
     \begin{matrix}
  \xymatrix@C=-4.5mm@R=10mm{
  Y\!\ar[drr]_{\pi^{}_{U, Y}}&\hskip 3mm\supseteq \hskip -10mm \!&\An_0&\hskip -10mm\subseteq\hskip 3mm&\hskip -1mm\An\ar@{-->}[dll]^{\pi^{}_{U,\An}}\ar@{-->}@/^2.3pc/[ddll]^{\pi^{}_{G,\An}}\\
  &&Y/U=\An\dss U\ar@{->}[d]_{\pi^{}_{H, Y/U}}&&\\
 &&(Y/U)\dss H=\An\dss G&&
  }
  \end{matrix}
  \end{equation}

For every element $c\in k$, denote by $L_c$ the hyperplane in $\An$ defined by the equation $x_d=c$.
The set
\begin{equation*}
\Omega:=\{c\in k\mid L_c\cap \An_0\neq\varnothing\}
\end{equation*}
is nonempty and open in $k$.\;Take an element $c\in \Omega$ and a point $a\in \An_0$. By property (ii) above, $s$ is regular and does not vanish at $a$. Consider the $U$-orbit of $a$ in $Y$. Formula \eqref{fd} shows that there is a unique $u_0\in U$ such that the value of $x_d\in k(Y)$ at $u_0\cdot a$ is $c$, namely,
\begin{equation}\label{param}
u_0=\frac{x_d-c}{s}(a).
\end{equation}
This means that
every $U$-orbit in $Y$ intersects $L_c\cap \An_0$ at most at one point, i.e., $\pi_{U, Y}|^{}_{L_c\cap \An_0}$ is injective.\;Since  $\dim L_c\cap \An_0=\dim Y/U$ and ${\rm char}\,k=0$,
this implies that
$\pi_{U, Y}|^{}_{L_c\cap \An_0}\colon L_c\cap \An_0\to Y/U$
 is a birational isomorphism.  Hence, $L_c$ intersects the domain of definition of
 $\pi^{}_{U,\An}$ and
 \begin{equation}\label{HHH}
\pi^{}_{U,\An}|^{}_{L_c}\colon L_c\dashrightarrow Y/U=\An\dss U,
 \end{equation}
 is a birational isomorphism. This means that $L_c$ is a rational cross-section for
 $\alpha|_U$.\;In particular, this
 implies that shrinking $\An_0$ if necessary, we may (and shall) assume that
 \begin{enumerate}[\hskip -4mm\rm(iii)]
\item for every point of $\An_0$, its $U$-orbit in $Y$  intersects $L_c$.
\end{enumerate}

Now we argue by induction on $\dim G$. If $\dim G=1$, then $G=U$ and the claim is proved
 since every $L_c$ for $c\in \Omega$ is a rational cross-section for $\alpha$ (so in this case $s=1$, $i_1=d$ and $\Theta_1=\Omega$).

 Now assume that $\dim G>1$.\;The action $\beta$ and the birational isomorphism \eqref{HHH}
determine a rational action $\gamma$ of $H$ on $L_c$ such that \eqref{HHH} becomes an $H$-equivariant birational isomorphism. From \eqref{diagra} we then deduce that
\begin{equation*}
\pi_{G,\An}|^{}_{L_c}\colon L_c\dashrightarrow \An\dss G
\end{equation*}
is a rational quotient for $\gamma$.
We identify $L_c$ with ${\bf A}^{n-1}$ by means of the isomorphism
$(a_1,\ldots, a_{d-1}, c, a_{d+1},\ldots, a_n)\mapsto (a_1,\ldots, a_{d-1}, a_{d+1},\ldots, a_n)$ and, for every function $f\in k(\An)$ whose domain of definition intersects $L_c$, put
\begin{equation*}
\overline{f\,}:=f|^{}_{L_c}\in k(L_c).
\end{equation*}
Then $\overline{x_1},\ldots,\overline{x_{d-1}},\overline{x_{d+1}},\ldots,\overline{x_n}$ are the standard coordinate functions on $L_c$.

We claim that the image of $H$ in ${\rm Aut}_kk(L_c)={\mathcal C}_{n-1}$ determined by the action $\gamma$ is contained in ${\mathcal J}(\overline{x_1},\ldots,\overline{x_{d-1}},\overline{x_{d+1}},\ldots,\overline{x_n})$. If this is proved, then, by the inductive assumption,   there exist  a nonempty set of indices $\,i_1,\ldots, i_r$ and a nonempty open subsets $\Theta_1,\ldots, \Theta_r$ of $\,k$\;such that
for every $(c_1,\ldots, c_r)\in \Theta_1\times\cdots\times\Theta_r$ the affine subspace $S$ of $L_c$ defined by the equations
 \begin{equation*}
 \overline{x_{i_1}}=c_1, \ldots, \overline{x_{i_r}}=c_r,
 \end{equation*}
 is a rational cross-section for $\gamma$, i.e.,
 \begin{equation*}
 \pi_{G, \An}|^{}_S\colon S\dashrightarrow \An\dss G
 \end{equation*}
 is a birational isomorphism.\;As $S$ is an affine subspace in $\An$ defined by the equations
 $x_d=c, x_{i_1}=c_1, \ldots, x_{i_r}=c_r$, this will complete the proof.

 It remains to prove the claim.\;To this end, consider in $k(\An)$ the subfield $k(\An)^U$
 of $U$-invariants elements with respect to
 $\alpha|_U$.\;Since $L_c$ is a rational cross-section of $\pi_{U, \An}$,
 the map
 \begin{equation*}\label{restri}
 k(\An)^U\to k(L_c),\qquad f\mapsto \overline{f\,},
 \end{equation*}
 is a well-defined
 $k$-isomorphism of fields
 that is
 $H$-equivariant with respect to
 the actions of $H$ on $k(\An)^U$  and $k(L_c)$ determined resp.\;by $\alpha$ and
 $\gamma$.
 Let
 \begin{equation*}
k(L_c)\to   k(\An)^U, \qquad t\mapsto\widehat t,
 \end{equation*}
 be the inverse isomorphism. Below we will consider $\beta$ and $\gamma$ as the actions of
 $G$ with the kernel $U$.

Take a point $a\in \An_0$.\;By the above discussion and property (iii), the $U$-orbit of $a$ in $Y$ intersects $L_c$ at a single point $u_0\cdot a$ where $u_0$ is given by \eqref{param}. As
$\widehat{\overline{x_i}\,}\in k(\An)^U$, this yields
\begin{equation}\label{==}
\widehat{\overline{x_i}\,}(a)=\widehat{\overline{x_i}\,}(u_0\cdot a)=\overline{x_i}(u_0\cdot a)=x_i(u_0\cdot a)=((-u_0)\cdot x_i)(a).
\end{equation}

Let $z, y_1,\ldots,y_{n-1}$ be the variables over $k$.\;It follows from \eqref{U}, \eqref{param} and  \eqref{==} that there are $d-1$ rational functions
\begin{equation*}
R_j(z, y_j, y_{j+1},\ldots, y_{n-1})\in k(z, y_j, y_{j+1},\ldots, y_{n-1}), \qquad j=1,\ldots, d-1,
 \end{equation*}
such that
\begin{equation}\label{R}
\widehat{\overline{x_i}\,}=
\begin{cases}x_i+R_i\Big(\displaystyle\frac{c-x_d}{s}, x_{i+1},\ldots, x_n\Big) &\mbox{if}\;\; i\leqslant d-1,\\
x_i&\mbox{if}\;\;i\geqslant d+1.
\end{cases}
\end{equation}
In turn, from \eqref{R}, \eqref{U}, \eqref{ff} and \eqref{ifl}  we infer that
\begin{equation*}\label{KKK}
g\cdot \widehat{\overline{x_i}\,}-x_i\in k(x_{i+1},\ldots, x_{d-1}, x_{d}, x_{d+1},\ldots, x_{n})\quad\mbox{for all $g\in G$}\;\;\mbox{and}\;\;i;
\end{equation*}
whence,
\begin{equation}\label{KK}
g\cdot \overline{x_i}-\overline{x_i}\in k(\overline{x_{i+1}},\ldots, \overline{x_{d-1}}, \overline{x_{d}}, \overline{x_{d+1}},\ldots, \overline{x_{n}})
\quad\mbox{for all $g\in G$}\;\;\mbox{and}\;\;i.
\end{equation}

The claim now follows from  \eqref{KK} because $\overline{x_{d}}=c\in k$.
\quad $\square$ \renewcommand{\qed}{}\end{proof}

\begin{corollary}\label{us}
For every unipotent algebraic subgroup $G$ of $\,{\rm GL}_n$, there exists an affine subspace $L$ of $\An$ such that $L$  is a rational cross-section for the natural action of $\,G$ on $\An$.
\end{corollary}
\begin{proof}
 There exists an element
$ g\in {\rm GL}_n$
 such that
 $gGg^{-1}\subset \mathcal J(x_1,\ldots, x_n)$ (see Example\;\ref{LK}).\;By Theorem \ref{unitri} there exists an affine subspace $S$ of $\An$ that is a rational cross-section for the natural action of $gGg^{-1}$ on $\An$.\;Then the affine subspace
$g^{-1}(S)$ is a rational cross-section for the natural action of $\,G$ on $\An$.
\quad $\square$ \renewcommand{\qed}{}\end{proof}

Here is the application of Corollary \ref{us}.\;Let $G$ be a connected affine algebraic group and let $\mathfrak g$ be the Lie algebra of $G$.\;Joseph put forward the following

\vskip 2mm

\noindent{\bf Conjecture
A\;{\rm
(\cite[Sect.\;7.11]{Jos11})}.} {\it For the coadjoint action of $\,G$ on ${\mathfrak g}^*,$ there exists an affine subspace $L$ of $\,\mathfrak g^*$ such that $k(\mathfrak g^*)^G\to k(L)$, $f\mapsto f|^{}_L$, is a well-defined isomorphism of fields.}

\vskip 2mm

Joseph calls such $L$ a {\it rational slice}.

According to the Levi decomposition, $\mathfrak g$ is a semidirect product of a reductive Lie algebra $\mathfrak r$ and the unipotent radical $\mathfrak u$,
\begin{equation}\label{Le}
\mathfrak g=\mathfrak r\ltimes \mathfrak u.
\end{equation}

\begin{corollary}
If \eqref{Le} is a direct product, $\mathfrak g=\mathfrak r\times \mathfrak u$,
then
Conjecture A
is true.
\end{corollary}

\begin{proof}  Let $R$ and $U$ be the closed connected subgroups of $G$ whose Lie algebras are resp.\;$\mathfrak r$ and $\mathfrak u$.\;Assume that $\mathfrak g=\mathfrak r\times\mathfrak u$.
In this case, if $L_{\mathfrak r}$ and
$L_{\mathfrak u}$ are  the rational slices for the coadjoint actions of resp.\;$R$ and $U$, then $L_{\mathfrak r}\times L_{\mathfrak u}$ is a rational slice for the coadjoint action of $G$.\;The existence of $L_{\mathfrak r}$ is proved in
\cite{Kos63} and  the existence of $L_{\mathfrak u}$ is ensured by Corollary\;\ref{us}.
\quad $\square$ \renewcommand{\qed}{}\end{proof}

 \begin{remark} {\rm Another application is that Theorem \ref{unitri} yields the results of 
 \cite[Part II, Chap.\,I, \S 7]{Pu}}.
 \end{remark}

The rational de Jonqui\`eres subgroup
$\mathcal J(t_1,\ldots, t_n)$ lies in another
interesting subgroup of  $\,{\mathcal C}_n$.\;Namely,
as for $\mathcal J(t_1,\ldots, t_n)$,
one checks that, for every $f_i\in K_i$ and $\mu_i\in K_i^\times$, there exists an element
$g\in\mathcal C_n$
 for which \eqref{T} holds and that the set $\widehat{\mathcal J}(t_1,\ldots, t_n)$ of all such elements
 $g$
 is a subgroup of $\mathcal C_n$.\;The flag of subfields \eqref{ff} is stable with respect to $\widehat{\mathcal J}(t_1,\ldots, t_n)$:
\begin{equation}\label{iflag}
g\cdot K_i= K_i\qquad\mbox{for all}\;\; g\in \widehat{\mathcal J}(t_1,\ldots, t_n)\;\;\mbox{and}\;\; i.
\end{equation}

If $s_1,\ldots, s_n$ is another system of generators of of $k(\An)$ over $k$, then the subgroups
$\widehat{\mathcal J}(t_1,\ldots, t_n)$ and $\widehat{\mathcal J}(s_1,\ldots, s_n)$
are conjugate in ${\mathcal C}_n$.

The following fact is known; it provides an information on tori in $\,\widehat{\mathcal J}(t_1,\ldots\linebreak\ldots , t_n)$ (see Corollary \ref{ttttt} below).
\begin{theorem}\label{TTT} For every  $($not necessarily algebraic\hskip .3mm$)$ subgroup
 $\,G$ of the group $\,\widehat{\mathcal J}(t_1,\ldots, t_n)$, the invariant field $k(\An)^G$ is pure over $k$.
 \end{theorem}
   \begin{proof} We shall sketch a proof since our argument
   provides a bit more
   information   (equality \eqref{KGx})
   than that of
   \cite{Miy71} and \cite{KV89}.\;The key ingredient
   is the following Miyata's lemma:
 \begin{lemma}[
 {\cite[Lemma]{Miy71}, cf.\,\cite[Lemme 1.1]{KV89},
 \cite[Sect.\;3]{AHK00}}]
 \label{Mi}
 Let $F$ be a field, let $z$ be a variable over $F$,  and let $H$ be a group that acts on
$F[z]$ by ring automorphisms
 leaving $F$ stable\footnote{It is not assumed that $F$ is pointwise fixed.}.\;Then the subfield of $F(z)$ generated by $F(z)^H$ over $F$  is, in fact, generated
 by a single element
 $x\in F[z]^H:$
 \begin{equation}\label{Kx}
 F(F(z)^H)=F(x).\qquad \square
 \end{equation}
 \end{lemma}

 Turning to the proof of Theorem \ref{TTT}, we first show that, in the notation of Lemma\;\ref{Mi},
 \begin{equation}\label{invar}
F(z)^H=F^H(x).
 \end{equation}
 Indeed, $F(x)^H\subseteq F(z)^H$ since $F(x)\subseteq F(z)$.\;On the other hand, \eqref{Kx} entails that
 $F(z)^H\subseteq F(x)^H$. Hence, $F(z)^H= F(x)^H$. Therefore, \eqref{invar} would be proved if
 the equality
 \begin{equation}\label{KGx}
 F(x)^H=F^H(x)
 \end{equation}
 is established.\;To prove \eqref{KGx}, consider two cases:\;(a) $x\in F$,  (b) $x\notin F$.\;If (a) holds, then $F(x)=F$, hence, $F(x)^H=F^H$.\;On the other hand, (a) and $x\in F[z]^H$\;yield that $x\in F^H$, hence, $F^H(x)=F^H$.\;This proves \eqref{KGx} if (a) holds.\;Now assume that (b)\;holds. Then $x$ is transcendental over $F$ by
 \cite[\S73, Theorem]{vdW67}.\;Consider an element $f\in F(x)^H$.\;It can be written as $f=p/q$ where
 \begin{equation}\label{polyn}
 p=\sum_{i=0}^sa_ix^i,\qquad q=\sum_{j=0}^r b_jx^j,\quad a_i, b_j\in F, \quad a_sb_r\neq 0,
 \end{equation}
and $p$ and $q$ are relatively prime po\-ly\-no\-mi\-als in $x$ with the coefficients in $F$.  Since $F[x]$ is a factorial ring, the relative primeness of $p$ and $q$ and $H$-invariance of $f$ imply that there is a map $\gamma\colon H\to F^*$ (in fact, a 1-cocycle) such that
\begin{equation}\label{gam}
h\cdot p=\gamma(h)p,\qquad h\cdot q=\gamma(h)q\quad\mbox{for every}\;\;h\in H.
\end{equation}
Since $x$ is $H$-invariant, \eqref{polyn} and \eqref{gam} yield
\begin{equation}\label{coeff}
h\cdot a_i=\gamma(h)a_i,\qquad h\cdot b_j=\gamma(h)b_j\quad \mbox{for all}\;\;h\in H \;\;\mbox{and}\;\;i, j.
\end{equation}
From \eqref{coeff} we infer that $f\!=\!a_s^{-1}p/a_s^{-1}q\!\in\! F^H(x)$.\;Thus, $F(x)^H\!\subseteq\! F^H(x)$. Since\;$x$ is $H$-invariant, the inverse inclusion is clear.\;This proves \eqref{KGx}.\;Thus, \eqref{invar} holds and, moreover,
either $x\in F^H$ or $x$ is transcendental over $F$.

Now let $G$ be a subgroup of  $\,\widehat{\mathcal J}(t_1,\ldots, t_n)$.\;We have $K_{i-1}=K_{i}(t_i)$ and  $t_i$ is transcendental over $K_i$ for every $i=1,\ldots, n$.\;By \eqref{iflag} and the definition of  $\,\widehat{\mathcal J}(t_1,\ldots, t_n)$ the action of $G$ on $K_{i-1}$ satisfies the conditions of Lemma \ref{Mi} (with $F=K_i$, $z=t_i$, $H=G$).\;Hence, as is proved above, there is an element
$z_i\in K_i[t_i]^G$ such that $K_{i-1}^G=K_{i}(t_i)^G=K_{i}^G(z_i)$ and either $z_i\in K_i^G$ or $z_i$ is transcendental over $K_i$. Respectively, either $K_{i-1}^G=K_i^G$ or $K_{i-1}^G$ is pure over $K_i^G$ of transcendental degree\;1.\;Since
\begin{equation*}
k=K_n^G\subseteq K_{n-1}^G\subseteq\cdots\subseteq K_1^G\subseteq K_0^G=k(\An)^G,
\end{equation*}
this completes the proof.
 \quad $\square$ \renewcommand{\qed}{}\end{proof}

 \begin{corollary}\label{ttttt} Every torus
 in  $\,\widehat{\mathcal J}(t_1,\ldots, t_n)$ is conjugate in $\,\mathcal C_n$ to a subgroup of\;$D_n$.
 \end{corollary}
 \begin{proof}
This follows from Theorems  \ref{torus} and \ref{TTT}.
 \quad $\square$ \renewcommand{\qed}{}\end{proof}

\begin{corollary}\label{fifi} Let $n\geqslant 5$.\;Every $(n-3)$-dimensional connected affine algebraic group
can be realized as an algebraic subgroup of $\,\mathcal C_n$ such that
\begin{enumerate}[\hskip 9mm]
\item[$({\rm i})$] $G$ is not conjugate to a subgroup of $\,\widehat{\mathcal J}(t_1,\ldots, t_n)${\rm;}
\item[$({\rm ii})$] the natural rational action of $\,G$ on $\An$ is locally free.
\end{enumerate}
\end{corollary}
\begin{proof} This follows from Theorems \ref{nr} and \ref{TTT}.
\quad $\square$ \renewcommand{\qed}{}\end{proof}

 \begin{remark} {\rm The assumption that $k$ is algebraically closed is not used in the proof of Theorem \ref{TTT}.\;}
 \end{remark}

 \begin{remark} {\rm  In
 \cite{Miy71}, Lemma \ref{Mi} is used for proving that $k(\An)^G$ is pure over $k$ if $G$ is a subgroup of ${\rm GL}_n\cap {\mathcal J}(x_1,\ldots, x_n)$.\;Note that in this case, if $G$ is finite, then purity of $k(\An)^G$ over $k$ follows from
 Corollary \ref{FFFF} and Lemma\;\ref{diag}.}
 \end{remark}

 \begin{remark} {\rm A weakened version of
 Theorem \ref{TTT}
  is
  the subject of
  \cite{Vin92}. In\;it,\;$G$ is an affine algebraic group and $\widehat{\mathcal J}(x_1,\ldots, x_n)$ is replaced by ${\mathcal J}(x_1,\ldots\linebreak\ldots, x_n)$.
  How\-ever,
  the argument
  in
  \cite{Vin92}
 does  not amount to complete and accurate proof.\;Indeed,
it is based on the claim, left unproven, that if $G$ is reductive, then $G$ is conjugate in ${\mathcal J}(x_1,\ldots, x_n)$ to
 a subgroup of $D_n$.\;Further, the claim that, for  a one-dimensional
 unipotent algebraic group $U$, ``every point is $U$-equivalent to a unique point of the subspace $S=\{x\in k^n : x_m=0\}$'' is false because
 $u\cdot s$ may be not defined for
 $u\!\in\! U$ and $s\!\in\! S$.\;Ditto for the claim  that $F_i\!\in\! k(x_{i+1},\ldots, x_n)\otimes k[t]$
 (counterexample: $n=3$ and the action  is given by
 $t\cdot x_1=x_1-t/x_2(x_2+t)$,\;$t\cdot x_2=x_2+t$,\;$t\cdot x_3=x_3$), so the equality $F_m(x_{m+1},\ldots, x_n; t)=tF_m(x_{m+1},\ldots, x_n)$
 remains
 unproven.}
\end{remark}

\begin{remark}\label{Trian}
{\rm
 One cannot replace $\,\widehat{\mathcal J}(t_1,\ldots, t_n)$ in Theorem \ref{TTT} by the
$\mathcal C_n$-stabilizer of the flag of subfields  \eqref{ff}. Indeed, by
\cite{Tri80}, for $k=\mathbf C$, $n=3$, this stabilizer contains a subgroup $G$ of order $2$ such that $k(\An)^G$ is not pure over $k$.
}
\end{remark}

Combining the construction from
\cite{Tri80} with Corollary \ref{stT} and Lemma\;\ref{diag} we obtain the following

\begin{theorem} Let  $k=\mathbf C$ and let $\mathcal A$ be the union of all
con\-nec\-ted affine algebraic subgroups of $\;\mathcal C_\infty$.\;There exists an element
$g\in{\mathcal C}_3$
of order $2$ such
that
$g\notin \mathcal A$.\;In particular,
$g$ is not stably li\-nea\-rizable.
\end{theorem}
\begin{proof}
Let $X$ be the three-dimensional counterexample of Artin and Mumford to the L\"uroth problem 
(\cite{AM72}, see\;also
\cite{Del70}): $X$ is a smooth projective unirational threefold such that
\begin{equation}\label{H3}
{\rm H}^3(X, \mathbf Z)_{\rm tors}\neq 0.
\end{equation}
Since the torsion subgroup of
the third integral cohomology group of a smooth complex variety is a birational invariant and, in
particular, is zero if the variety is rational,
\eqref{H3} implies that $X$ is not rational.

  In
  \cite{Tri80} is constructed a subgroup $G$ of order $2$ in $\mathcal C_3$ such that $k(\mathbf A^3)^G$ is $k$-isomorphic to
$k(X)$.
Let $g$ be the generator of $G$.\;Arguing on the contrary, assume that $g$ is contained in a connected affine algebraic subgroup $H$ of $\mathcal C_\infty$.\;Since the order of $g$ is finite, $g$ is a semisimple element of $H$.\;Hence,  $g$ lies in a maximal torus $T$ of $H$ (see
\cite[Theorems III.10.6(6) and IV.11.10]{Bor91}).\;By Corollary \ref{stT} there exists a positive integer $n_0$ such that
$T\subset \mathcal C_{n_0}$ and $T$ is conjugate in $\mathcal C_{n_0}$ to a subtorus of $D_{n_0}$.\;Fix an integer $n\geqslant {\rm max}\{n_0, 3\}$.\;Then  $G\subset \mathcal C_n$
and $G$ is conjugate in $\mathcal C_n$
to a subgroup of $D_n$.\;This and Lemma\;\ref{diag} yield that for the natural action of $G$ on $\mathbf A^n$ the field $k(\An)^G$ is pure over\;$k$.
Since $G\subset \mathcal C_3$, by
\cite[Lemma 3]{Ros612}, we have
\begin{equation}\label{XP}
\An\dss G\approx \mathbf A^3\dss G\times \mathbf A^{n-3}\approx X\times \mathbf P^{n-3}
\end{equation}
From \eqref{XP} we infer that the smooth projective variety $X\times \mathbf P^{n-3}$ is rational and therefore ${\rm H}^3(X\times \mathbf P^{n-3})_{\rm tors}
=0$.\;On the other hand, the K\"unneth formula and  \eqref{H3} yield that ${\rm H}^3(X\times \mathbf P^{n-3})_{\rm tors}\neq 0$\,---\,a contradiction.
\quad $\square$ \renewcommand{\qed}{}\end{proof}


We conclude by an example which shows
that  in the formulation of  Corollary\;\ref{us}
``unipotent'' cannot be replaced by ``connected solvable'' (recall that if $G$ is connected solvable, then the existence of {\it some} rational cross-section is ensured by
\cite[Theorem\;10]{Ros56}).

\begin{example}
Fix a choice of two integers $d_1$ and $d_2$ such that
\begin{gather}
d_1-d_2\geqslant 2, \label{sub}\\[-2pt]
|d_1|\geqslant 2,
\quad
 |d_2|\geqslant 2,\label{d}
\\[-2pt]
{\rm gcd}(d_1, d_2)=1.\label{gcd}
\end{gather}
Consider
 the one-dimen\-sio\-nal sub\-to\-rus
\begin{equation}\label{act}
T:=\big\{{\rm diag}\big(t^{d_1}, t^{d_2}\big)\mid t\in k^\ast\big\}
\end{equation}
of  $D_2$ and its rational linear action $\alpha$ on ${\bf A}\hskip -.5mm^2$ defined by formula \eqref{diago}.

In view of  \eqref{gcd}, the $T$-stabilizer of every point $a\in {\bf A}\hskip -.5mm^2$, $a\neq (0, 0)$, is trivial.

\vskip 2mm

 \noindent {\bf Claim}.\;{\it There is no affine subspace in ${\bf A}\hskip -.5mm^2$ that is a rational cross-sec\-tion for\;$\alpha$.}

 \vskip 2mm

 \begin{proof} Assume that some affine subspace $L$ of ${\bf A}\hskip -.5mm^2$  is a rational cross-section for $\alpha$. Since $T$-orbits in general position are one-dimensio\-nal,
$L$ is a line. Let
\begin{equation}\label{eq}
\mu_1x_1+\mu_2x_2+\nu=0, \qquad \mu_1,\mu_2,\nu\in k.
\end{equation}
be its equation.\;Since $L$ is a rational cross-section, there is a nonempty open subset $U$ of ${\bf A}^{\hskip -.5mm 2}$
such that for every point $a=(a_1, a_2)\in U$, the $T$-orbit of $a$ intersects $L$ at a single point, i.e., by \eqref{act} and \eqref{eq}, the following equation in~$t$
\begin{equation}\label{root}
\mu_1a_1t^{d_1}+\mu_2a_2t^{d_2}+\nu=0
\end{equation}
has a single nonzero solution. Shrinking $U$, we may assume that $a_1a_2\neq 0$ for
every $a\in U$.

If $\mu_1\mu_2=0$, then \eqref{root} becomes an equation  of the form $\mu t^d+\nu=0$ where $\mu\in k$, $\mu\neq 0$, and $|d|\geqslant 2$ by  \eqref{d}. If $\nu=0$, it does not have nonzero solutions; if $\nu\neq 0$, there are at least two such solutions. So this case is impossible.

If $\mu_1\mu_2\neq 0$ and $\nu=0$, then the solutions of \eqref{root} coincide with the roots of
$\mu_1a_1t^{d_1-d_2}+\mu_2a_2$.\;In view of  \eqref{sub}, there are at least two distinct roots, so this case is impossible as well.

Let $\mu_1\mu_2\nu\neq 0$ and $d_2>0$. Denote by $f$ be the right-hand side of \eqref{root}.\;Set
\begin{equation}\label{h}
h:=d_1\hskip -.3mm f-t\frac{df}{dt}=(d_1-d_2)\mu_2a_2t^{d_2}+d_1\hskip -.2mm\nu.
\end{equation}
By \eqref{root} and \eqref{h}, for a fixed $a_2$, there are only finitely many $a_1$'s such that the polynomials $f$ and $h$ have a
common root.\;Since every multiple root of $f$ is also a root of $h$, this means that there are
points $a\in U$ such that
$f$ does not have multiple roots.\;From \eqref{sub}, \eqref{d} it then follows that for such a point $a$ equation
\eqref{root}
has at least two  nonzero solutions.\;Thus, this case is also impossible.

Finally, let  $\mu_1\mu_2\nu\neq 0$ and $d_2<0$.\;Then the solutions of equation \eqref{root} coincide with the roots of the polynomial
$q:=\mu_1a_1t^{d_1-d_2}+\nu t^{-d_2}+\mu_2a_2$. We have
\begin{equation}\label{last}
p:=(d_1-d_2)q-t\frac{dq}{dt}=(d_1-d_2)\mu_2a_2+d_1\nu t^{-d_2}.
\end{equation}
Then the same argument as above with $f$ and $h$ replaced resp.\;by
$q$ and $p$ shows that
this case is impossible as well.

This contradiction completes the proof.
\quad $\square$
\renewcommand{\qed}{}\end{proof}
\end{example}



\begin{thebibliography}{CTKPR11}

\bibitem[AHK00]{AHK00} {\sc H. Ahmad, M. Hajja, M.-c. Kang}, {\it Rationality of some 
projective linear actions}, J. Algebra {\bf 228} (2000), 643--658.


\bibitem[AM72]{AM72}  {\sc M.\;Artin, D.\;Mumford}, {\it Some elementary examples of unirational varieties which are not rational}, Proc. London Math. Soc.\,(3) {\bf 25} (1972), 75--95.

 \bibitem[BCSS85]{BCSS85} {\sc A.\;Beauville, J.-L.\;Colliot-Th\'el\`ene, J.-J.\;Sansuc, P.\;Swinnerton-Dyer}, {\it Varietes stablement rationneles non rationneles}, Annals of Math. {\bf 121} (1985), no.\;2, 283--318.

\bibitem[Bia66]{Bia66}  {\sc A.\;Bia{\l}ynicki-Birula}, {\it Remarks on the action of an algebraic torus on $k^n$}, {\rm Bull. Acad. Polon. Sci. S\'er. Sci. Math., Astr., Phys.} {\bf 14} (1966), no. 4, 177--181.

\bibitem[Bor91]{Bor91}
{\sc A.~Borel}, {\em Linear Algebraic Groups}, Graduate Texts in Mathematics,
  Vol.~126, 2nd edn. (Springer-Verlag, New York, 1991).


\bibitem[BS64]{BS64} {\sc A.\;Borel,\;J.-P.\;Serre}, {\it   Th\'eor\`emes de finitude en cohomogie galoisienne}, Comm. Math. Helv. {\bf 39} (1964), 111--164.


\bibitem[Che54]{Che54} {\sc C.\;Chevalley}, {\it On algebraic group varieties}, J. Math. Soc. Japan {\bf 6} (1954), nos. 3--4, 303--324.

     \bibitem[CTKPR11]{CTKPR11}{\sc J.-L.\;Colliot-Th\'el\`ene, B.\;Kunyavski\v \i, V.\;L\;Popov,
 Z.\;Reichstein}, {\it Is the function field of
 a reductive Lie algebra purely
 transcendental over the field of
 invariants for the adjoint action}?,\;
 Compositio Math.\;{\bf 147}\;(2011),\;no.\;02,\;428--466.

 \bibitem[Del70]{Del70}
{\sc P.~Deligne}, {\em Vari\'et\'es unirationnelles non rationnelles $($d'apr\`es M.
  Artin et D. Mumford$)$}, in {\em S\'eminaire Bourbaki\/}, Lecture Notes in
  Mathematics, Vol.\,317, Exp.\,no.\,402 (Springer-Verlag, Berlin, 1973), pp.\,45--57.

\bibitem[Dem70]{Dem70} {\sc M.\;Demazure}, {\it Sous-groupes alg\'ebriques de rang maximum du groupe de Cremona}, Ann. Sci. de l'E.N.S., 4e s\'er. {\bf 3} (1970), fasc.\,4, 507--588.

    \bibitem[Gro58]{Gro58}
{\sc A.~Grothendieck}, {\em Torsion homologique et sections rationnelle}, in {\em
  Anneaux de Chow et Applications\/}, S\'eminaire Claude Chevalley,\,Vol.\,3,\,Exp.\,no.\,5 (Secr\'etariat math\'ematique, Paris, 1958), pp.\,1--29.


\bibitem[Hal59]{Hal59}
{\sc M.~Hall, Jr.}, {\em The Theory of Groups} (Macmillan, New York, 1959).

\bibitem[Hir72]{Hir72}
{\sc A.~Hirschowitz}, {\em Le groupe de Cremona d'apr\`es Demazure}, in {\em
  S\'eminaire N. Bourbaki\/}, Lecture Notes in Mathematics,\,Vol.\,317,\,Exp.\,no.\,413 (Springer-Verlag, Berlin, 1973), pp.\,261--276.



  \bibitem[Jos11]{Jos11} {\sc A.\;Joseph}, {\it An algebraic slice in the coadjoint space of the Borel and the Coxe\-ter ele\-ment}, Adv. Math. {\bf 227} (2011), no.\,1, 522--585.


 \bibitem[KV89]{KV89} {\sc M.\;Kervaire, T.\;Vust}, {\it   Fractions rationnelles invariantes par un groupe fini: Quelqu\-es exemples}, in: {\it Algebraic Transformation Gropus and Invariant Theory}, DMV Se\-mi\-nar, Band 13, Birkh\"auser, Basel, 1989, pp.\;157--179.

\bibitem[Kos63]{Kos63} {\sc B.\;Kostant},
{\it Lie group representations on polynomial rings},
{\rm Amer. J. Math.} {\bf 85} (1963), 327--404.

      \bibitem[Mer80]{Mer80}
{\sc Y.~I. Merzl{ya}kov}, {\em Rational Groups} (Nauka, Moscow, 1980) (in Russian).

 \bibitem[Miy71]{Miy71} {\sc T.\;Miyata}, {\it Invariants of certain groups} I,
 Nagoya Math. J. {\bf 41} (1971), 69--73.





  \bibitem[Pop72]{Pop72} {\sc V.\;L.\;Popov}, {\it On the stability of the action of an algebraic group on an algebraic variety}, Math. USSR Izv. {\bf 6} (1972), 367--379 (1973).

\bibitem[Pop94]{Pop94}
{\sc V.~L. Popov}, {\em Sections in invariant theory}, in\,{\em Proceedings of The
  Sophus Lie Memorial Conference, Oslo $1992$\/}
  (Scandinavian University Press, Oslo, 1994), pp. 315--362.

\bibitem[PV94]{PV94}
{\sc V.~L. Popov, E.~B. Vinberg}, {\em Invariant theory}, in\,{\em Algebraic
  Geometry} IV, Encyclopaedia of Mathematical Sciences, Vol.\,55
  (Springer-Verlag, Berlin, 1994), pp. 123--284.


\bibitem[Pu]{Pu}
{\sc L.~Pukanszky}, {\em Lec\c{o}ns sur les Repr\'esentations des Groupes},
  Monographies de la Soci\'et\'e Math\'ematique de France, Vol.\,2 (Dunod,
  Paris, 1967).

  \bibitem[Ros56]{Ros56} {\sc M.\;Rosenlicht}, {\it Some basic theorems on algebraic groups},
Amer. J. Math. {\bf 78} (1956), 401--443.

   \bibitem[Ros61$_1$]{Ros611} {\sc M.\;Rosenlicht}, {\it Toroidal algebraic groups},
    Proc. Amer. Math. Soc. {\bf 12} (1961), 984--988.

\bibitem[Ros61$_2$]{Ros612} {\sc M.\;Rosenlicht}, {\it On quotient variaties and the affine embedding of certain homogeneous spaces}, Trans. Amer. Math. Soc. {\bf 101} (1961), 211--223.

    \bibitem[Ser58]{Ser58}
{\sc J.-P. Serre}, {\em Espaces fibr\'es alg\'ebriques}, in {\em Anneaux de Chow et
  Applications\/}, S\'eminaire Claude Chevalley,\,Vol.\,3,\,Exp.\,no.\,1
  (Secr\'etariat math\'ematique, Paris, 1958), pp.\,1--37.

\bibitem[Ser08]{Ser08}
{\sc J.-P. Serre}, {\em Le groupe de Cremona et ses sous-groupes finis}, in {\em
  S\'eminaire N. Bourbaki}, Vol.\,$2008/2009$, Ast\'erisque Vol.\,332, Exp.\,no.\,1000 (Soci\'et\'e Math\'e\-matique de France, 2010), pp.\,75--100.

  \bibitem[Sh-B04]{SB04} {\sc N.-I. Shepherd-Barron}, {\it
   Stably rational irrational varieties}, in: {\it The Fano Conference},  Univ. Torino, Turin, 2004, pp.\;693--700.

\bibitem[Sum74]{Sum74} {H.\;Sumihiro}, {\it Equivariant completion}, J. Math. Kyoto Univ. {\bf 14} (1974), no.\;1, 1--28.

\bibitem[Tri80]{Tri80} {\sc D.\;D.\;Triantaphyllou}, {\it Invariants of finite groups acting nonlinearly on rational function fields}, J. Pure Appl. Algebra {\bf 18} (1980), 315--331.

\bibitem[vdE00]{vdE00}
{\sc A.~van~den Essen}, {\em Polynomial Automorphisms}, Progress in Mathematics,
  Vol.~190 (Birkh{\"a}user, Basel, 2000).

  \bibitem[vdW67]{vdW67}
{\sc B.~L. van~der Waerden}, {\em Algebra} {\rm I} (Springer-Verlag, Berlin, 1967).

\bibitem[Vin92]{Vin92} {\sc E.\;B.\;Vinberg}, {\it Rationality of the field of invariants of a triangular group}, Mosc. Univ. Math. Bull. {\bf 37} (1992), no.\,2, 27--29.

\bibitem[Wei55]{Wei55} {\sc A.\;Weil}, {\it On algebraic groups of transformations}, Amer. J. Math. {\bf 77} (1955), no.\,2, 355--391.

\end{thebibliography}

\end{document}